\documentclass[a4paper,10pt,leqno, english]{smfart}
\usepackage{aeguill}
\usepackage{enumerate}
\usepackage{amssymb,amsmath,latexsym,amsthm}
\usepackage[T1]{fontenc}
\usepackage{smfthm}      
\usepackage{geometry}   
\usepackage{url}
\usepackage[frenchb, english]{babel}
\usepackage[utf8]{inputenc}   
\usepackage{mathrsfs}
\usepackage{xcolor}
\usepackage{comment}
\definecolor{violet}{rgb}{0.0,0.2,0.7}
\definecolor{rouge2}{rgb}{0.8,0.0,0.2}
\usepackage{etoolbox}
\patchcmd{\thebibliography}{*}{}{}{}
\pretocmd\thebibliography{\csname c@secnumdepth\endcsname=-2 }{}{}
\usepackage{hyperref}
\usepackage{relsize}

\hypersetup{
    bookmarks=true,         
    unicode=false,          
    pdftoolbar=true,        
    pdfmenubar=true,        
    pdffitwindow=false,     
    pdfstartview={FitH},    
    pdftitle={},    
    pdfauthor={},     
    colorlinks=true,       
   linkcolor=violet,          
    citecolor=rouge2,        
    filecolor=black,      
    urlcolor=cyan}           
\numberwithin{equation}{section}

\renewcommand{\P}{\mathbb{P}}
\newcommand{\R}{\mathbb{R}}
\newcommand{\CC}{\mathbb{C}}
\newcommand{\Q}{\mathbb{Q}}

\renewcommand{\d}{\partial}

\newcommand{\vp}{\varphi}

\renewcommand{\O}{\mathcal{O}}
\newcommand{\Ox}{\mathcal{O}_{X}}
\newcommand{\ep}{\varepsilon}
\renewcommand{\epsilon}{\varepsilon}

\newcommand{\la}{\langle}

\newcommand{\ra}{\rangle}

\renewcommand{\ge}{\geqslant}
\renewcommand{\le}{\leqslant}
\renewcommand{\leq}{\leqslant}
\renewcommand{\geq}{\geqslant}
\newcommand{\Ric}{\mathrm{Ric} \,}
\newcommand{\om}{\omega}
\newcommand{\omi}{\omega_{\infty}}

\newcommand{\ddc}{dd^c}
\newcommand{\xreg}{X_{\rm reg}}

\newcommand{\bi}{\beta_{\infty}}

\newcommand{\E}{\mathscr{E}}
\newcommand{\F}{\mathscr{F}}

\newcommand{\omvp}{\omega_{\varphi}}

\newcommand{\Supp}{\mathrm {Supp}}
\newcommand{\tr}{\mathrm{tr}}

\newcommand{\pp}{\psi^{+}}
\newcommand{\psm}{\psi^{-}}
\newcommand{\ppm}{\psi^{\pm}}

\newcommand{\Xnu}{X^{\nu}}
\newcommand{\Dnu}{D^{\nu}}

\newcommand{\Tx}{\mathscr{T}_X}
\newcommand{\Txd}{\mathscr{T}_X(-D)}
\newcommand{\txdf}{\mathscr{T}_X(-D)}
\newcommand{\Txdf}{T_X(-D)}

\renewcommand{\F}{\mathscr{F}}
\newcommand{\G}{\mathscr{G}}
\newcommand{\omx}{\om_{X}}
\newcommand{\omy}{\om_{Y}}

\newcommand{\Hom}{\mathrm{Hom}_B}
\newcommand{\T}{\mathcal T}
\newcommand{\Lm}{L^{\otimes m}}
\newcommand{\Km}{K_X^{\otimes m}}
\newcommand{\rk}{\mathrm{rk}}
\newcommand{\hr}{h^{\wedge r}}
\newcommand{\Exc}{\mathrm{Exc}}
\newcommand{\pr}{\mathrm{pr}}

\newtheorem*{thma}{Theorem A}
\newtheorem*{thmb}{Theorem B}

\newtheorem*{KH}{Kobayashi-Hitchin correspondence}
\newtheorem*{thme}{Corollary}

\renewcommand{\theequation}{\thesection  .\!\! \arabic{equation}}

\begin{document}

\title{Semi-stability of the tangent sheaf of singular varieties}

\date{\today}
\author{Henri Guenancia}
\address{Department of Mathematics \\
Stony Brook University \\
 Stony Brook, NY 11794-3651 USA}
\email{guenancia@math.sunysb.edu}
\urladdr{http://www.math.sunysb.edu/~guenancia}

\begin{abstract}
The main goal of this paper is to prove the polystability of the logarithmic tangent sheaf $\mathscr T_X(-D)$ of a log canonical pair $(X,D)$ whose canonical bundle $K_X+D$ is ample, generalizing in a significant way a theorem of Enoki. We apply this result and the techniques involved in its proof to get a version of this theorem for stable varieties (the higher dimensional analogue of Deligne-Mumford's stable curves) and to prove the polystability with respect to any polarization of the tangent sheaf of a singular Calabi-Yau variety.
\end{abstract}

\maketitle
\tableofcontents
\section*{Introduction}
\subsection*{Semi-stability}
This paper deals with the notion of slope semi-stability for coherent sheaves on singular varieties using a differential-geometric approach. This notion was first introduced by Mumford \cite{Mum63} in his attempt to construct bounded families of vector bundles over a curve. Recall that a vector bundle $E$ over a complex projective curve $C$ is said to be semistable (resp. stable) if for every proper subbundle $F$ of $E$, we have
\[\frac{\mathrm{deg}(F)}{\rk \,F}\le \frac{\mathrm{deg}(E)}{\rk \,E} \quad  \left(\mathrm{resp.} \,\, \frac{\mathrm{deg}(F)}{\rk \,F}<\frac{\mathrm{deg}(E)}{\rk \,E} \right)\]
where the degree of $F$ is $\int_C c_1(F)$; the quantity $\mathrm{deg}(F)/\rk\, F$ is called the slope of $F$ and denoted $\mu(F)$. When one tries to generalize this definition to higher dimensional varieties one first needs an ample line bundle $H$ (called a polarization) to define a degree and thus a slope: $\deg_H(F)=c_1(F) \cdotp H^{n-1}$ and $\mu_H(F)= \mathrm{deg}_H(F)/\rk \, F $. And more importanty, in order to keep the good moduli properties shared by stable vector bundles over curves, one needs to ask that not only all the subbundles of $E$ have a smaller slope than the one of $E$, but also all its proper coherent \textit{subsheaves}, cf Definition \ref{defi:slope}.

\subsection*{The Kobayashi-Hitchin correspondence}
One of the major developments regarding stable vector bundles is the celebrated Kobayashi-Hitchin correspondence, relating the stability of a vector bundle $E$ (an algebro-geometric property) to the existence of a special hermitian metric on $E$ (a differential-geometric property): 

\begin{KH}
Let $E$ be an holomorphic vector bundle on a compact Kähler manifold $(X, \om)$. Then $E$ is polystable with respect to $\om$ if and only if $E$ admits a Hermite-Einstein metric with respect to $\om$.
\end{KH}

Recall that a Hermite-Einstein metric $h$ on $E$ is a Hermite-Einstein metric with respect to a Kähler form $\om$ if its Chern curvature tensor $\Theta_h(E)\in \mathscr C^{\infty}(X, \Omega^{1,1}\otimes \mathrm{End}(E))$ satisfies $\tr_{\om}( \Theta_h(E)) = \mu  \mathrm{Id}_E$ for some constant $\mu \in \R$. Here, polystability means that $E$ is semistable and is the sum of stable subsheaves (that automatically have same slope).

\noindent
This correspondence is due to Kobayashi-Lübke \cite{Kob82, Lub} in the direction "Hermite-Einstein $\Rightarrow$ stable" and to Donaldson for Riemann surfaces \cite{Don83}, algebraic surfaces \cite{Don85} and manifolds \cite{Don87} and to Uhlenbeck-Yau \cite{UY, UY2} for Kähler manifolds. 

The case where $E=T_X$ is the tangent bundle of the manifold is of particular importance because due to the symmetry of the Riemann curvature tensor, if $\om$ is a Kähler form and $h$ is the induced hermitian metric on $T_X$, then $(T_X, h)$ is Hermite-Einstein with respect to $\om$ if and only if $\om$ is a Kähler-Einstein metric, i.e. $\Ric \om = \mu \om$ for some constant $\mu \in \R$. For a lot of reasons, the Kähler-Einstein equation is in general easier to study/solve than the Hermite-Einstein equation, and currently we dispose of a lot of existence results and tools in the Kähler-Einstein theory. For example, when $K_X$ is ample (resp. numerically trivial), the celebrated theorems of Aubin-Yau and Yau \cite{Aubin, Yau78} provide a unique negatively curved Kähler-Einstein metric (resp. Ricci-flat metric) on $X$ living in $c_1(K_X)$ (resp. in any Kähler cohomology class). Bringing these theorem together with the Kobayashi-Hitchin correspondence, we see that:

\begin{thme}
Let $X$ be a compact Kähler manifold. Then
\begin{enumerate}
\item[$\bullet$] If $K_X$ is ample, then $T_X$ is polystable with respect to $K_X$.
\item[$\bullet$] If $K_X$ is numerically trivial, then $T_X$ is polystable with respect to any Kähler class.
\end{enumerate}
\end{thme}

\subsection*{Singular varieties and stability} In the present paper, we will strive to extend the above result in a singular setting pushing further the very elegant approach of I. Enoki \cite{Enoki} who essentially proved the semistability of the tangent sheaf of $X$ in both cases above assuming merely that $X$ has canonical singularities. His strategy is to use approximate Kähler-Einstein metric on a resolution $X'$. By considering them as approximate Hermite-Einstein metrics on $T_{X'}$, he obtains an "approximate slope inequality". Thanks to various estimates (due to Yau) about Monge-Ampère equations, he manages to control the error term in the inequality and to prove that it converges to zero at the end of the approximation process.

We will generalize Enoki's Theorem from three different perspectives: we will merely require that $X$ has log canonical singularities, we will allow a boundary divisor and work in the the setting of log pairs, and finally we will show the polystability on top of the semistability:

\begin{thma}
Let $X$ be a compact Kähler space and $D$ be a reduced Weil divisor. 
\begin{enumerate}
\item[$\bullet$] If $(X,D)$ has log canonical singularities and $K_X+D$ is ample, then the logarithmic tangent sheaf $\mathscr T_X(-D)$ is polystable with respect to $K_X+D$.
\item[$\bullet$] If $X$ has klt singularities and $K_X$ is numerically trivial, then $T_X$ is polystable with respect to any Kähler class.
\end{enumerate}
\end{thma}

Compared to the proof in the case of canonical singularities, a lot of serious new difficulties pop up in this general setting, mainly because the presence of boundary divisor $D$ requires to work with cuspidal metrics instead of smooth metrics. Also passing from canonical to log canonical singularities depends upon a finer analysis of the regularity for solutions of degenerate Monge-Ampère equations, which we develop in \S \ref{sec:lap}. We outlined the new issues we had to overcome in \S \ref{statement}.

We should mention that the second item of Theorem A is closely linked with the main result of \cite{GKP} which describes the tangent sheaf of varieties with canonical singularities and vanishing first Chern class. Also, some of the techniques used in the proof appear in the very interesting paper \cite{CP2} where one of the clever ideas to treat boundary problems is to use conic metrics (with angles going to $0$) in order to avoid the use of cuspidal metrics.

\subsection*{Stable varieties}
The first item of Theorem A above can be used to deduce properties about the so-called stable varieties. For short, stable varieties are the higher dimensional analogue of stable curves introduced by Deligne and Mumford, in the sense that they allow to compactify the moduli space of canonically polarized manifolds (say with fixed Hilbert polynomial). Their definition involves three requirements: projectivity, positivity of the canonical bundle and specific singularities, cf \S \ref{sec:stable0} and the references therein for a more detailed account. For the time being, we just need to know that if $X$ is a stable variety, then $K_X$ can be defined as a ample $\Q$-line bundle, and the normalization map $\nu:\Xnu\to X$ satisfies that $\nu^*K_X=K_{\Xnu}+D$ where $D$ is a reduced Weil divisor such that the pair $(\Xnu, D)$ has log canonical singularities. 

Therefore, Theorem A guarantees that $\mathscr T_{\Xnu}(-D)$ is semistable with respect to $K_{\Xnu}+D$, and one can easily infer from  
this that $\mathrm{Aut}(X)$ is a finite group, cf Corollary \ref{cor:stable}. One can also ask whether on can say anything about semistability of the tangent sheaf of $X$ itself. Inspired by the case of stable curves, we introduce the sheaf $\nu_*\mathscr T_{\Xnu}(-D)$ that should be the "tangent sheaf" we want to look at in this setting where the philosophy is to describe the objects upstairs. For non-normal varieties, the notions of slope and stability exist, but are way more delicate to deal with. We recall them in \ref{sec:stable}, and manage to prove the following result which is a non-trivial consequence of Theorem A:

\begin{thmb}
Let $X$ be a stable variety, and $\nu:\Xnu\to X$ its normalization map. Then sheaf $\nu_*\mathscr T_{\Xnu}(-D)$ is semistable with respect to $K_X$.
\end{thmb}

The proof of this theorem requires to understand the relationship between the slope of a sheaf and the one of its pull-back under the normalization map, which happens to be trickier than one could imagine at first sight, cf Remark \ref{rm:ct}. 

\subsection*{Organization of the paper}
\begin{itemize}
\item[$\bullet$] \S \ref{sec:lap}: We obtain a Laplacian estimate for solutions of Monge-Ampère equations with degenerate right hand side, which will turn out to be crucial for the proof of the semistability.
\item[$\bullet$] \S \ref{sec:poly}: We first recall the definitions of stability, log canonical singularities, logarithmic tangent sheaf, then we state the first part of Theorem A, explain the difficulties and give its proof.
\item[$\bullet$] \S \ref{sec:stable}: We give the definition of stable varieties as well as semistability in a context of non-normal varieties, and then prove Theorem B.
\item[$\bullet$] \S \ref{sec:gsp}: We prove a result of generic semipositivity (Theorem \ref{thm:gsp}) in the spirit Miyaoka's Theorem, that we apply to prove the second item of Theorem A.\\
\end{itemize}

\noindent
\textbf{Acknowledgements.} This work was initiated near the end of my PhD thesis, and I am grateful to my advisors Sébastien Boucksom and Mihai P\u{a}un who generously shared their ideas and helped me develop this article. I am particularly indebted to Mihai P\u{a}un who shared his unpublished work \cite{CP2} with me and followed the development of this paper with interest. Finally, I would also like to thank Patrick Graf for helpful discussions.

\section{A Laplacian estimate}
\label{sec:lap}
This section is devoted to the technical Proposition \ref{prop:est} which can be seen as a generalization of the Laplacian estimate obtained using Chern-Lu formula. It will be used in the next section to control the error term in the semistability inequality.

So let us first recall Chern-Lu's formula \cite{Chern, Lu}, which is going to be a essential tool to get the laplacian estimate. Let $(X,\om_X)$ and $(Y, \om_Y)$ be two Kähler manifolds, and $f:X\to Y$ an holomorphic map satisfying $\partial f\neq 0$. Then 
\[  \Delta_{\omx} \log |\partial f|^2 \ge \frac{\Ric \omx \otimes \omy(\d f, \overline {\d f})}{|\partial f|^2}-\frac{ \omx \otimes R^Y(\d f, \overline {\d f},\d f, \overline {\d f})}{|\partial f|^2} \]
where $\d f$ is viewed as a section of $T^*_X \otimes T_Y$.\\

Using this formula when $f$ is the identity map (but the Kähler forms differ), one can derive so-called laplacian estimates for the Kähler-Einstein equation provided that the reference metric has upper bounded holomorphic bisectional curvature (cf \cite[Section 7]{JMR}). The following proposition, inspired by \cite{Paun} (see also \cite[Theorem 10.1]{BBEGZ}), enables to derive (weaker) laplacian estimates in some cases where the Ricci curvature is not bounded from below:

\begin{prop}
\label{prop:est}
Let $X$ be a compact Kähler manifold of dimension $n$, $\om, \om'$ two cohomologous Kähler metrics on $X$. We assume that $\om'=\om+\ddc \vp$ with $\om'^n=e^{\pp-\psm} \om^n$ for some smooth functions $\ppm$, and that we have a constant $C>0$ satisfying:
\begin{enumerate}
\item[$(i)$]  $\sup_X |\vp| \le C$,
\item[$(ii)$] $\sup_X \pp \le C \, \,\mathrm{and} \, \, \, \ddc \ppm \ge -C \om$,
\item[$(iii)$]  $  \Theta_{\om}(T_X) \le C \om \otimes \mathrm{Id}_{T_X}$. 
\end{enumerate}
Then there exists some constant $M>0$ depending only on $n$ and $C$ such that 
\[ \om' \ge M^{-1} e^{\pp} \om.\]
\end{prop} 

\begin{proof}

The main difficulty is that we do not have a control on the lower bound of $\Ric \om'= \Ric \om + \ddc \psm-\ddc \pp$. The trick, inspired by \cite{Paun}, is to add $\pp$ in the laplacian appearing in Chern-Lu formula. Let us now get into the details. 
We apply Chern-Lu's formula to $f=\mathrm{id}: (X, \om') \to (X, \om)$. Then $|\partial f|^2 = \tr_{\om'}\om$. We denote by $g$ (resp. $h$) the hermitian metrics induced by $\om'$ (resp. $\om$). The second term of Chern-Lu formula is easily dealt with: 
\begin{eqnarray*}
 \om' \otimes R^{h}(\d f, \overline {\d f},\d f, \overline {\d f})&=& -g^{i\bar j}g^{k \bar l} R_{i \bar j k \bar l}^h\\
 & \ge& -C g^{i\bar j}g^{k \bar l} (h_{i\bar j}h_{k \bar l}+h_{i\bar l}h_{k \bar j}) \\
 & \ge & -2C (\tr_{\om'}\om)^2
 \end{eqnarray*}
 Now recall that if $\alpha$ and $\beta$ are two hermitian $(1,1)$ forms, and if $(\alpha, \beta)_{\om'}$ denotes the hermitian product induced by $\om'$, then we have
 \[(\alpha, \beta)_{\om'}= \tr_{\om'} \alpha \, \cdotp \, \tr_{\om'} \beta -n(n-1) \frac{\alpha \wedge \beta \wedge {\om'}^{n-2}}{{\om'}^{n}} \]
Moreover, one can check that \[\alpha\otimes \beta( \partial f, \overline{\partial f}) = (\alpha,\beta)_{\om'}\] (in the tensor product, $\alpha$ stands for the hermitian form induced by $\alpha$ on $T_X^*$ relatively to $\om'$). 
As $\Ric \om' \ge -nC \om -(C \om +\ddc \pp)$, using the two previous identities, we get: 
\[\Ric \om' \otimes \om(\d f, \overline {\d f}) \ge -nC (\tr_{\om'}\om)^2 - \tr_{\om'}\om \, \cdotp \, \tr_{\om'}(C\om+\ddc \pp)\] 
At that point, Chern-Lu formula gives us: 
\[\Delta_{\om'} \log \tr_{\om'}\om \ge  -(n+3)C \,  \tr_{\om'}\om - \tr_{\om'}\ddc \pp\]
and therefore: 
\[\Delta_{\om'}( \log\tr_{\om'}\om+\pp) \ge  -(n+3)C \,  \tr_{\om'}\om\]
Setting $A=(n+3)C+1$, we get as usual:
\[\Delta_{\om'}( \log \tr_{\om'}\om+\pp-A\vp) \ge  \tr_{\om'}\om-nA\]
The end is classic: we choose a point $p$ where $\log \tr_{\om'}\om+\pp-A\vp$ attains its maximum; then we have
\begin{eqnarray*}
\log \tr_{\om'}\om & \le & ( \log \tr_{\om'}\om+\pp-A\vp)(p)-\pp+A \vp\\ 
& \le & (\log nA+ \sup \pp +2A \sup |\vp|) - \pp
\end{eqnarray*}
which gives the expected result since we have a uniform bound on $||\vp||_{\infty}$ by assumption. 
\end{proof}

\begin{rema}
Note that it is not clear a priori to deduce the previous estimate using \cite{Paun} by exchanging the role of $\om$ and $\om'$ because we would no longer have control on the bisectional curvature of $\om'$.   
\end{rema}
Combining the previous result and P\u{a}un's estimate, we obtain the following estimate: 

\begin{coro}
Let $X$ be a compact Kähler manifold of dimension $n$, $\om, \om'$ two cohomologous Kähler metrics on $X$. We assume that $\om'=\om+\ddc \vp$ with $\om'^n=e^{\pp-\psm} \om^n$ for some smooth functions $\ppm$, and that we have a constant $C>0$ and some $p>1$ satisfying:
\begin{enumerate}
\item[$(i)$] $\sup_X \pp \le C \, \,\mathrm{and} \, \, ||e^{-\psm}||_{L^p(\om^n)} \le C$,
\item[$(ii)$] $ \ddc \ppm \ge -C \om$.
\end{enumerate}
Then there exists some constant $M>0$ depending only on $n,p, C$ and $\om$ such that 
\[ M^{-1} e^{\pp} \om \le \om' \le  M e^{-\psm} \om.\]
\end{coro}

\section{Polystability of the logarithmic tangent sheaf}
\label{sec:poly}
\subsection{Generalities}

In this section, we briefly recall the definitions of slope, semi-stability, various types of singularities appearing in the Minimal Model Program as well as the (logarithmic) tangent sheaf before exposing the previously known results. We refer to \cite{Har2}, \cite[Chap. V]{Koba} or \cite{HL} for more details. 

\subsubsection{Notion of stability}
In the following, $X$ will be a complex projective normal variety of dimension $n$, and $\F$ will always denote a coherent sheaf. We write $\F^*= \mathscr {H}\! om(\F, \Ox)$ for the dual of $\F$. We say that $\F$ is reflexive if the natural map \[j : \F \to \F^{**} \]
is an isomorphism. For instance, the dual of a coherent sheaf is always reflexive.  Moreover, as the kernel of $j$ is exactly the torsion of $\F$, a reflexive sheaf is automatically torsion-free. 
We define the rank of a coherent sheaf $\F$ to be its rank at the generic point (or equivalently, consider the Zariski open subset where $\F$ is locally free), we denote if by $\mathrm{rk} \,  \F$. \\

We need now to define the determinant of a coherent sheaf $\F$. We set $r= \mathrm{rk} \, \F$, and we let \[\det \F := (\Lambda^r \F)^{**}\] be the determinant of $\F$. It is a rank one reflexive sheaf on $X$, but if $X$ is not smooth, it is in general not a line bundle, ie it is not locally free. 
As $X$ is normal, there is a $1-1$ correspondence between rank one reflexive sheaves (up to isomorphism) and Weil divisors (up to linear equivalence), the correspondence being given in the usual way $D \mapsto \Ox(D)= \{f, \mathrm{div} ( f) \ge -D\}$ in one direction, and in the other direction, given $\F$ a rank one reflexive sheaf, we choose a Weil divisor on $\xreg$ representing the line bundle $\F_{|\xreg}$ and take its closure. 

\noindent
We will denote by $c_1(\F)$ the equivalence class of any Weil divisor attached to $\det \F$. We can now define the slope:

\begin{defi}
\label{defi:slope}
Let $H$ be an ample line bundle on $X$. We define the slope of $\F$ with respect to $H$ to be the rational number
\[\mu_H(\F):= \frac{c_1(\F) \cdotp H^{n-1}}{\mathrm{rk} \, \F}\]
\end{defi}

Let us get now to the definition of slope stability, which goes back to Mumford \cite{Mum63} and Takemoto \cite{Take}:

\begin{defi}
Let $\mathscr E$ be a torsion-free coherent sheaf on $X$, and $H$ be an ample line bundle. 

\noindent
We say that $\mathscr E$ is \textit{semistable} (resp. \textit{stable}) with respect to $H$ if for every coherent subsheaf $\F \subset \mathscr E$ (resp. every non-zero and proper coherent subsheaf $\F$), we have
\[\mu_H(\F) \le \mu_H(\mathscr E) \quad \mathrm{(resp.} \, \mu_H(\F) < \mu_H(\mathscr E))\]
We say that $\mathscr E$ is \textit{polystable} (with respect to $H$) if $\mathscr E$ is the direct sum of stable subsheaves with same slope.
\end{defi}

\subsubsection{Canonical and log canonical singularities}
Let us now very briefly give a definition of the class of singularities we are going to deal with (see \cite{KM} for a more detailed account): 

\begin{defi}
\label{sing}
Let $X$ be a complex normal variety, $D$ an \textit{effective} Weil divisor $D$ such that $K_X+D$ is $\Q$-Cartier and $\pi:X'\to X$ a log resolution of $(X, D)$. We define the coefficients $a_i$ by the formula $$K_{X'}=\pi^*(K_X+D)+\sum a_iE_i$$ 
where $E_i$ is either an exceptional divisor or the strict transform of a component of $D$. 

$\cdotp$ If $D=0$, we say that $X$ has canonical (resp. log terminal) singularities if for all $i$, one has $a_i \ge 0$ (resp. $a_i>-1$)

 $\cdotp$ Else, we say that the pair $(X,D)$ has klt (resp. log canonical) singularities if for all $i$, one has $a_i > -1$ (resp. $a_i\ge-1$)
\end{defi}

\noindent
So all these singularities are normal so that the notions of slope and semistability make sense for these varieties. \\

\subsubsection{The logarithmic tangent sheaf}
Let us first recall that on an arbitrary complex variety $X$, one can define the sheaf $\Omega_X^1$ of Kähler differentials (cf \cite[II.8]{Har}) and define the tangent sheaf $\mathscr T_X$ of $X$ as the dual $(\Omega_X^1)^*$ of the sheaf of Kähler differentials. If $X$ is smooth, then this sheaf corresponds to the sheaf associated with the tangent bundle $T_X$. If $X$ is merely normal, then $\mathscr T_X$ is a reflexive sheaf so that by the observation above, it can equivalently be defined as push-forward of $\mathscr T_{\xreg}$ via the open immersion $\xreg \hookrightarrow X$.\\

Let us now consider log pairs: let $X$ be a normal projective variety, and $D$ be a reduced Weil divisor on $X$. If the pair $(X,D)$ is log smooth, we have a well defined logarithmic tangent bundle $T_{X}(-D)$ which is the dual of the bundle of logarithmic differentials $\Omega_X^1(\log D)$. It simply consists of vector fields that vanish along $D$, i.e. if $D$ is locally given by $(z_1\cdots z_k=0)$, then the sheaf at stake is the locally free $\mathcal O_X$-module generated by $$z_1 \frac{\d}{\d z_1}, \ldots, z_k \frac{\d}{\d z_k}, \frac{\d}{\d z_{k+1}}, \ldots, \frac{\d}{\d z_n}$$

To define its analogue in a more general setting, we proceed as follows.
We denote by $(X,D)_{\rm reg}$ the simple normal crossing (snc for short) locus of the pair $(X,D)$, i.e.  the locus of points $x\in X$ where $(X,D)$ is log smooth in a neighborhood of $x$. As $X$ is normal, it is smooth in codimension $1$, and $D$ is generically smooth, therefore $(X, D)_{\rm reg}$ is a Zariski open set whose complement has codimension at least $2$. Let us denote by $j :(X,D)_{\rm reg} \hookrightarrow X $ the open immersion. 

\begin{defi}
Let $(X, D)$ be a log pair as above, and denote by $U:=(X,D)_{\rm reg}$ its snc locus. The logarithmic tangent sheaf of $(X,D)$ is defined as $j_*T_{U}(-D_{|U})$. 
\end{defi}

This sheaf is automatically coherent, and hence reflexive by e.g. \cite[Proposition 1.6]{Har2}. One can equivalently define the logarithm tangent sheaf as the sheafification of the module of derivations that preserve the ideal sheaf corresponding to $D$.

\subsection{Statement of the result}
\label{statement}
We have now all the tools in hand to explain our results. We showed in the introduction how the conjunction of the Kobayashi-Hitchin correspondence 
and the Aubin-Yau theorem provided the polystability of $T_X$ whenever $X$ is a compact Kähler manifold such that $K_X$ is ample or trivial. Relying on this and the robustness of the Kähler-Einstein theory, Enoki \cite{Enoki} managed to essentially extend this result whenever $X$ has \textit{canonical} singularities in the sense given in Definition \ref{sing} above. His strategy consists of working on a resolution and constructing there smooth approximate Kähler-Einstein metrics. They induce approximate Hermite-Einstein metrics on $T_X$ (in a naive sense), and using the fact that the singularities are canonical, he manages to control the error terms or at least the one arising with a bad sign in view of the semistability property.\\

Our main result extends Enoki's theorem (say when $K_X$ is ample) to the case of a log canonical pair $(X,D)$ and also provides the polystability:

\begin{theo}
\label{thm:pair}
Let $X$ be a normal projective variety and $D$ be a reduced Weil divisor such that  $K_X+D$ is an ample $\Q$-line bundle. Then $\Txd$ is polystable with respect to $K_X+D$. 
\end{theo}

Compared to Enoki's setting, a large amount of new difficulties arise in this generalized setting that require a lot of fine analysis and some recent works: 
\begin{enumerate}
\item[$\cdotp$] Getting precise estimates for Monge-Ampère with a very degenerate right hand side (this is the content of Lemma \ref{lem2} which is a combination of Proposition \ref{prop:est} and \cite[Theorem A]{GW}) 
\item[$\cdotp$] Working with cuspidal metrics instead of smooth ones (which requires the delicate Proposition \ref{slope1})
\item[$\cdotp$] Analyzing the limiting behavior of the approximate Kähler-Einstein metrics (we will use \cite{BG} and \cite[Theorem B]{GW} in a crucial way) to get polystability
\end{enumerate}

\noindent
However, these complications come from the relatively large degree of generality that we chose to work with, and we could have obtained weaker generalizations of Enoki's theorem at a less expansive cost, cf Remark \ref{remarque}.\\

\subsection{An application to automorphism groups}
A first application of Theorem \ref{thm:pair} concerns the finiteness of some automorphism groups. Indeed, it is well-known that semistable sheaves with negative slopes do not admit any non-trivial section, and relying on this fact and the above theorem, one can prove:

\begin{coro}
\label{cor:stable}
The automorphism group $\mathrm{Aut}(X)$ of a stable variety $X$ is finite. 
\end{coro}

Here, a variety is said to be stable (in the sense of Koll\'ar-Shepherd-Barron-Alexeev) if it is projective, has semi-log canonical singularities, and has an ample canonical bundle, cf Section \ref{sec:stable0}. The finiteness of the automorphism group of such varieties was already well-known, cf \cite{Masa, BHPS} or \cite[Theorem 6.3]{BG} for a differential-geometric proof in the normal case.

\begin{proof}[Proof of Corollary \ref{cor:stable}]
Let $\nu:Y\to X$ be the normalization of $X$. By the universal property of the normalization, any automorphism of  $X$ can be lifted to a automorphism of $Y$ giving an injection $i:\mathrm{Aut}(X)\to \mathrm{Aut}(Y)$; observe that  any element in $i(\mathrm{Aut}(X))$ preserves the conductor $\Delta$ of the normalization. Moreover, the subgroup of $\mathrm{Aut}(Y)$ fixing every connected component of $Y$ has finite index in $\mathrm{Aut}(Y)$, so we can assume that $Y$ is connected. 

So we are reduced to showing that given a normal variety $Y$ and a reduced divisor $\Delta$ such that the pair $(Y,\Delta)$ has log canonical singularities and satisfies $K_Y+\Delta$ is ample, then the automorphism group $\mathrm{Aut}(Y,\Delta)$ of the pair $(Y,\Delta)$ is finite; this group is defined as the (closed) subgroup of $\mathrm{Aut}(Y)$ fixing $\Delta$.
But any such automorphism will preserve $K_Y+\Delta$, and therefore our group is a subgroup of the automorphism of a polarized variety so it is a linear group (as it embeds into $\mathrm{PGL}(N,\mathbb C)$ for some large integer $N$). Therefore $\mathrm{Aut}(Y,\Delta)$ is finite if and only if its tangent space is trivial, but its tangent space is precisely $H^0(Y, \mathscr T_Y(-\Delta))$, the space of holomorphic vector fields tangent to $\Delta$, cf e.g. \cite[Lemma 5.2]{BBEGZ}.

We are almost done as any non-zero element $\xi \in H^0(Y, \mathscr T_Y(-\Delta))$ would generate a trivial rank $1$ subsheaf of  $\mathscr T_Y(-\Delta)$, hence having vanishing slope. By the semistability of $\mathscr T_Y(-\Delta)$, its slope (with respect to $K_Y+\Delta$) should then be non-negative, whereas it is equal to $-(K_Y+\Delta)^n/n<0$.
\end{proof}

\subsection{Proof of the Theorem}

\begin{proof}[Proof of Theorem \ref{thm:pair}]
We proceed in four main steps.\\

\noindent
\textbf{Step 1. Reduction to the log smooth case}\\
Let $(Y,\Delta)$ be a log canonical pair with $\Delta$ a reduced Weil divisor such that $K_Y+\Delta$ is ample, and let $\G$ be a coherent subsheaf of $\mathscr{T}_Y(-\Delta)$ of rank $r>0$. We choose $\pi: X \to Y$ a log resolution of the pair that is an isomorphism over its snc locus; we denote by $\Delta'$ the strict transform of $\Delta$. There exists a $\pi$-exceptional divisor $E=\sum a_i E_i$ such that $K_X+\Delta'=\pi^*(K_Y+\Delta)+E$, and we have $a_i \ge -1$ for all $i$ by the log canonical assumption. Let us set $F=\sum_{a_i=-1}E_i$ and $D:=\Delta'+F$; for some reasons that will appear later, we want to find a coherent subsheaf $\F$ of $\txdf$ such that $\pi_*c_1(\F)=c_1(\G)$. We consider the snc locus $(Y, \Delta)_{\rm reg}$ of the pair, which open subset of $Y$ whose complement has codimension at least $2$, and we set $U:=\pi^{-1}((Y, \Delta)_{\rm reg})$ and let $j:U\to X$ denote the open immersion. Then the sheaf $\F:=(j_{*}(\pi^{*}\mathscr G)_{|U}) \cap \txdf$ is a coherent subsheaf of $\txdf$ that satisfies $\pi_*c_1(\F)=c_1(\G)$ on $(Y, \Delta)_{\rm reg}$, hence on the whole $Y$.\\
Finally, as $c_1(\mathscr{T}_Y(-\Delta))$ is represented by $-(K_Y+\Delta)$, the projection formula yields the following equivalences: 

\begin{eqnarray}
\label{slope2}
\frac{c_1(\G)\,\cdotp \, (K_Y+\Delta)^{n-1}}{r} & \le & \frac{c_1(\mathscr{T}_Y(-\Delta))\,\cdotp \,(K_Y+\Delta)^{n-1}}{n} \\
 \Longleftrightarrow \quad \frac{c_1(\F) \, \cdotp \, \pi^*(K_Y+\Delta)^{n-1}}{r} &\le &\frac{-\pi^*(K_Y+\Delta)^{n}}{n} \nonumber \\
 \Longleftrightarrow \quad  \frac{c_1(\F)\cdotp \pi^*(K_Y+\Delta)^{n-1}}{r} &\le & \frac{-(K_X+\Delta'-E) \, \cdotp \, \pi^*(K_Y+\Delta)^{n-1}}{n} \nonumber \\
 \Longleftrightarrow \quad  \frac{c_1(\F)\cdotp \pi^*(K_Y+\Delta)^{n-1}}{r} &\le & \frac{c_1(\txdf) \, \cdotp \, \pi^*(K_Y+\Delta)^{n-1}}{n} \nonumber
\end{eqnarray}
as $E$ is $\pi$-exceptional.\\
Therefore we are reduced to showing the semistability of $\txdf$ with respect to $\pi^*(K_Y+\Delta)$. Actually, we will prove a bit more and show that for every generically injective sheaf morphism $\F \to \txdf$, we have the expected slope inequality.\\

\noindent
\textbf{Step 2. Construction of appropriate cusp metrics}\\
Let us introduce some notations first. We choose $A$ an ample line bundle on $X$, $\om_A$ a Kähler form whose cohomology class is $c_1(A)$, and $\om_0$ a Kähler form on $Y$ representing $c_1(K_Y+\Delta)$. Recall that we denote by $F$ the "purely log canonical" part of $E$, ie $F=\sum_{a_i=-1}E_i$, we set $X_0:=X\setminus D$, and we fix two parameters $\ep,t>0$ for our regularization process. Finally, we consider for each $i$ such that $a_i>-1$ a regularizing family $(\theta_{i,\ep})_{\ep>0}$ of $(1,1)$-forms in $c_1(a_iE_i)$ such that $\theta_{i,\ep}$ converges to the singular current $a_i[E_i]$ when $\ep$ goes to zero. An explicit formula for $\theta_{i,\ep}$ can be given in terms of hermitian metrics $h_i$ on $\Ox(E_i)$ as well as sections $s_i$ cutting out the divisors $E_i$. In the following, we will work with the metric $\om$ satisfying:
\begin{equation}
\label{eq:ricci}
\Ric \om = -\om +t\om_A+[D]-\sum_{a_i>-1} \theta_{i,\ep}
\end{equation}

This metric $\om$ belongs to $c_1(\pi^ *(K_Y+\Delta)+tA)$ and depends on $\ep$ and $t$ but we choose not to mention its dependence to keep the notation lighter. It is smooth on $X_0$, and has cusp singularities along $D$, ie around points where $D$ is given by $(z_1 \cdots z_r=0)$, $\om$ is uniformly equivalent to the model cusp metric $\sum_{k=1}^r \frac{i dz_k \wedge d\bar z_k}{|z|^2 \log^2 |z|^2}+\sum_{k>r} idz_k \wedge d \bar z_k$. The existence (and uniqueness) of $\om$ is guaranteed by results of Kobayashi \cite{KobR} and Tian-Yau \cite{Tia}, who obtained it as the solution of the Monge-Ampère equation 
\[(\pi^ *\om_0+t\om_A+\ddc \vp)^ n = \frac{e^ {\vp+f}\prod_{i} (|s_i|^ 2+\ep^ 2)^ {a_i}}{\prod_j |t_j|^2}\]
where $f$ is some smooth function determined by the hermitian metrics chosen on $E$ and $D$, and $(s_i=0)$ (resp. $(t_j=0)$) cuts out $E_i$ for $i$ such that $a_i>-1$ (resp. the $j$-th component of $D$). Moreover, one can rewrite this equation in a perhaps more standard form using the potential $\vp_P =-\sum_j \log \log^2 |t_j|^2$ of the cusp metric (also called Poincaré metric) along $D$: setting $\om_P:=\pi^*\om_0+t\om_A+\ddc\vp_P$ and $\psi:=\vp-\vp_P$, the equation becomes
\begin{equation}
\label{eq:cusp}
(\om_P+\ddc \psi)^ n= \prod_{i} (|s_i|^ 2+\ep^ 2)^ {a_i} e^ {\psi+F} \om_P^ n
\end{equation}
where $F$ is a smooth function when read in the quasi-coordinates, cf  \cite{KobR} or \cite[Lemma 4.3]{G12}. The metric $\om=\om_P+\ddc \psi$ (that depends strongly on $\ep$ and $t$) satisfies the following property, which will be a cornerstone of the proof:

\begin{lemm}
\label{lem2}
For every fixed $t>0$, and any section $s$ of a component of $D+E$, the integral
\[\int_{X_0} \frac{\ep^ 2}{|s|^2+\ep^ 2} \, \om_A \wedge \om^ {n-1}\]
converges to $0$ when $\ep$ goes to $0$. 
\end{lemm}

\begin{proof}
Assume for the moment that we can prove that there is a constant $C$ depending only on $t>0$ such that 
\begin{equation}
\label{eq:min}
\om \ge C^ {-1} \prod_{a_i>0} (|s_i|^ 2+\ep^ 2)^ {a_i} \om_P
\end{equation}
In view of the equation \eqref{eq:cusp} satisfied by $\om$ and the above assumption, we have
\[\tr_{\om} (\om_A) \, \om^n \le C \, \frac{\prod_{a_i>-1}(|s_i|^2+\ep^2)^{a_i} e^{\psi} \om_P^n}{\prod_{a_i>0} (|s_i|^2+\ep^2)^{a_i}}\]
for some $C>0$ under control (depending on $\sup |F|$ and a constant $M>0$ such that $\om_A \le M \om_P$). We need here another non-trivial input, as we want to get rid of $\psi$. This is possible by invoking \cite[Theorem A]{GW} (or more precisely its proof) that provides a bound $\sup |\psi| \le C$ independent of $\ep$ (but depending on $t$). Combining these observations, we get: 
\[\frac{\ep^2}{|s|^2+\ep^2} \, \om_A \wedge \om^{n-1} \le  \frac{C\ep^2 }{|s|^2+\ep^2} \cdot  \prod_{-1<a_i<0} (|s_i|^2+\ep^2)^{a_i} \om_P^n \]
Checking the convergence to $0$ of our integral is now local on $\Supp(D+E)$. As this divisor has simple normal crossing support, we can use Fubini's theorem to reduce our problem to a one-dimensional one. We will be done once we've prove that both integrals
\[\int_{\mathbb D} \frac{\ep^2 \, i dz\wedge d\bar z}{(|z|^2+\ep^ 2)^{1+\delta}} \quad \mathrm{and} \quad \int_{\mathbb D} \frac{\ep^2 \, i dz\wedge d\bar z}{(|z|^2+\ep^ 2) \cdotp |z|^2 \log^2 |z|^2} \]
converge to $0$, for any $\delta \in (0,1)$, $\mathbb D$ being a small disc centered at $0 \in \CC$.  
For the first integral, we can perform the change of variable $w=z/\ep$, we get the integral \[\ep^{2(1-\delta)}\int_{|w|^2 \le 1/\ep} \frac{|dw|^2}{(1+|w|^2)^{1+\delta}} \]
which goes to zero as $1>\delta>0$. As for the second one, we know that $\ep^ 2 (|z|^2+\ep^ 2)^{-1} \le 1$, and that $(|z|^2 \log^2 |z|^2)^{-1} \in L^1(\mathbb D)$; we then just have to apply the dominated convergence theorem to conclude. 

The only thing left to prove now is \eqref{eq:min}. As \cite{GW} provides us with an $L^ {\infty}$ estimate on $\psi$, this inequality would follow from Proposition \ref{prop:est} if we could establish it for metrics $\om, \om'$ without the compactness assumption. It turns out that in our situation, the reference metric is $\om_P$ has bounded curvature tensor, and the unknown metric $\om$ is complete with (qualitatively) bounded Ricci curvature (even bounded holomorphic bisectional curvature) so that we can apply the generalized maximum principle of Yau and mimic the proof of Proposition \ref{prop:est} without any serious change. This ends the proof of \eqref{eq:min}, and hence of the lemma. \\
\end{proof}

\noindent
\textbf{Step 3. Computing the slopes using the singular metrics}\\
We start with a generically injective morphism $\F \to \txdf$. If $r=\mathrm{rk}(\F)$, then this morphism induces $j:\det \F\longrightarrow \Lambda^r \txdf$ where by definition, $\det \F:=(\Lambda^r \F)^{**} $. The strategy consists of endowing $T_X(-D)$ (as well as $\F$) with the hermitian metric induced by $\om$ and computing the slopes of these sheaves using the corresponding representatives of their first Chern classes. Two different types of difficulty appear though: first, as $\om$ is cuspidal along $D$, it induces a singular hermitian metric on $T_X(-D)$, and it is not clear that one can use it to compute any slope (whether for $T_X(-D)$ or $\F$). And secondly, as $\om$ is not Kähler-Einstein but only approximately, it is not clear that the error term in the slope inequality will vanish when the various regularizing parameters (essentially $\ep$ and $t$) will converge to $0$.\\

We start by addressing the first difficulty. In order to do so, we introduce the singularity set $W=W(\F)$ which is defined as the smallest analytic subset of $X$ outside which the sheaf morphism $\F\to \txdf$ is an injection of vector bundles. We denote by $F$ the vector bundle on $X\setminus W$ such that $\F=\Ox(F)$ there. Without loss of generality, one can assume that $\F$ is saturated in $\txdf$ (it is enough to consider such subsheaves to test the semi stability), so that $W$ has codimension at least two in $X$. 

\noindent
We denote by $h$ the smooth hermitian metric induced by $\om$ on $T_X(-D)$ over $X\setminus D$. It induces also a smooth hermitian metric on $F$ over $X\setminus (W\cup D)$, that we still denote by $h$. We claim that $h$ can be used
to compute the slope of $\F$:

\begin{prop}
\label{slope1}
We have:
\[\int_{X\setminus(W\cup D)}c_1(F,h) \wedge \om^{n-1}=c_1(\F) \cdotp (\pi^ *(K_Y+\Delta)+tA)^{n-1}\]
\end{prop}

\begin{proof}
This is far from being obvious, especially because of the singularities of $\om$ along $D$. On $X\setminus(W\cup D)$, $h$ induces an hermitian metric $\hr$ on $\Lambda^rF$ that we aim to compare with a smooth hermitian metric $h_0$ on $\det \F$ \--- notice that $\Lambda^rF$ and $\det \F$ coincide on $X\setminus(W\cup D)$. We are now going to analyze the behavior of $\hr/h_0$ on $W\cup D$. More precisely, if $\xi$ denotes a local trivialization of the line bundle $\det \F$, then the quantity $\hr(\xi)/h_0(\xi)$ is independent of the choice of $s$ so it induces a smooth positive function $H$ on $X\setminus(W\cup D)$, and obviously $\ddc \log H = c_1(\det \F, h_0)-c_1(F,h)$, so everything boils down to showing that $\int_{X\setminus(W\cup D)}\ddc \log H \wedge \om^{n-1}=0$. For the sake of clarity, we will distinguish three cases, even if they could be treated in a unified manner:\\

$\bullet$ Near a point $x\in W\setminus D$, then the morphism $j:\det \F \to \Lambda^rT_X(-D)$ is degenerate. Choose trivialization of all bundles at stake. Then there are holomorphic function $(f_1, \ldots, f_p)$ where $p=C^r_n$ such that $j=(f_1,\ldots, f_p)$. Then $H=|j|_h^2$ is the squared norm (for some hermitian product on $\mathbb C^p$) of a vector of holomorphic functions, and we claim that $\log H$, possibly after a modification, is the sum of a smooth function and a pluriharmonic function outside $(j=0)$ \--- which is nothing else than $W$ intersected with our trivializing chart. To see that, we choose a log resolution $\nu$ of the ideal $(f_1, \ldots, f_p)$ which produces (above) a holomorphic function $g$ and a trivial ideal sheaf $(g_1, \ldots, g_p)$ such that $f_i=g_i g$ for each $i$. Therefore, $\pi^*\log H = \log \sum |g_i|^2+ \log |g|^2$, and the first summand extends smoothly across the exceptional divisor $(g=0)= \nu^{-1}(W)$. \\

$\bullet$ Near a point $x\in D\setminus W$, the hermitian metric $h$ is degenerate but $\F$ is a genuine subbundle of $T_X(-D)$. A local holomorphic frame of $T_X(-D)$ is given by $z_1 \frac{\d}{\d z_1}, \ldots, z_k \frac{\d}{\d z_k},  \frac{\d}{\d z_{k+1}}, \ldots,  \frac{\d}{\d z_n}$. Also, $\om$ is (possibly non uniformly in $t,\ep$ though) equivalent to the standard Kähler metric with cuspidal singularities along $D$. Therefore, in these coordinates $h$ is equivalent to the matrix
$h_{\rm cusp}:=\mathrm{diag}((-\log |z_1|^2)^{-2}, \ldots, (-\log |z_k|^2)^{-2}, 1, \ldots, 1)$. Also, near $x$, $\Lambda^r F$ is locally generated by one element $\xi \in \Lambda^r T_X(-D)$, 
\[\xi = \sum_{I=(i_1, \ldots, i_r)} a_I \prod_{\substack{i\in I\\ i\le k}}z_i \frac{\d}{\d z_{i_1}}\wedge \cdots \wedge \frac{\d}{\d z_{i_r}} \] where $I$ runs among all (unordered) $r$-tuples of $\{1, \ldots, n\}$. With respect to $\Lambda^r h_{\rm cusp}$, the squared norm of $\xi$ is equal to $\sum_I |a_I|^2 \prod_{i\in I\cap \{1, \ldots, k\}} (-\log |z_i|^2)^{-2}$. We choose a $r$-tuple $I$ such that $a_I(x)\neq 0$. One can assume that $a_I$ does not vanish up to shrinking the neighborhood we are working on. As $\hr$ is equivalent to $\Lambda^r h_{\rm cusp}$, we infer that 
\begin{equation}
\label{minor1}
\log |\xi |^2_{\hr} \ge \sum_{i\in I\cap \{1, \ldots, k\}} -\log \log^2 |z_i|^2-C
\end{equation}
for some given constant $C>0$. Remember that our ultimate goal is to show that $\int_{X\setminus(W\cup D)} \ddc \log H \wedge \om^{n-1}=0$; a way to prove this is to prove that $\log H$ is regular enough so that this integral is actually cohomological. Of course, $\log H$ is not smooth (or even bounded), but the above inequality tends to suggest that it has finite energy in the sense of \cite{GZ07}. But a zero order bound like this does not quite suffice to support this claim, as we also crucially need to bound its complex Hessian from below. As $\xi$ is a (local) holomorphic section over $X\setminus D$ of the hermitian bundle $(\Lambda^rT_X,h)$, one can write there: 
$$ \ddc \log |\xi|^2_{h} = \frac{1}{|\xi|^2_{h}}\left( |D'\xi|^2-\frac{|\la D'\xi, \xi\ra |^2}{|\xi|^2_{h}} \right) - \frac{\la\Theta_h(\Lambda^rT_X) \xi, \xi\ra}{|\xi|^2_{h}}$$
where $\Theta_h(\Lambda^r T_X)$ is the Chern curvature tensor of $(\Lambda^rT_X,h)$. As $|\la D'\xi, \xi\ra |^2\le  |D'\xi|^2\cdotp |\xi|^2$, the term in the parenthesis of the right hand side is non-negative. Moreover, we know from \cite{KobR, Tia} that $(X\setminus D, \om)$ has bounded geometry, so in particular $\om$ has bounded holomorphic bisectional curvature. Therefore the curvature tensor of $(\Lambda^r T_X, h)$ is bounded too, so that there is a constant $C$ (depending on $\ep$ and $t$) such that $\Theta_h(\Lambda^rT_X) \le C \om \otimes \mathrm{Id}_{\Lambda^r T_X}$ in the sense of Griffiths semipositivity, so in particular we get 
\begin{equation}
\label{minor2}
\ddc \log |\xi|^2_{h}\ge -C \om
\end{equation}
Still one cannot conclude yet that $\log H$ is in a finite energy class. Indeed, all we wrote requires $\xi$ not to vanish (so it applies as long as we are far away from $W$) and also the above inequality \eqref{minor2} does not say that $\log H$ is quasi-psh as $\om$ is certainly not dominated by a smooth Kähler form on $X$. The first issue will be addressed by using the methods of the first item above; as for the second, if we introduce $\vp_P:=\sum -\log \log^2 |s_i|^2$ where the $s_i$'s are sections cutting out the components of $D$ (and the norms being taken with respect to arbitrary smooth hermitian metrics), then $\ddc \vp_P$ is quasi-psh and there exists $\om_0$ a fixed smooth Kähler metric on $X$ such that $\om_0+\ddc \vp_P$ has cusp singularities along $D$. In particular there exists $C>0$ (depending on $t,\ep$ once again) such that $\om_0+\ddc \vp_P \ge C^{-1}\om$. Also, it follows from \cite[Proposition 2.3]{G12} that $\vp_P \in \mathcal E(X, \om_0)$, i.e. it has finite energy in the sense of \cite{GZ07}. 

If $W$ were empty, then we would have that for some big constant $C>0$, $\log H+C\vp_P$ is a quasi-psh function dominating a multiple of $\vp_P$ up to a constant, so it has finite energy too. Also, $\om$ is a finite energy current (its potential is equal to $\vp_P+O(1))$, therefore $\int_{X\setminus D} \ddc(\log H + C\vp_P) \wedge \om^{n-1}=0$. By linearity and as $\vp_P$ has finite energy, we deduce $\int_{X\setminus D} \ddc \log H \wedge \om^{n-1}=0$ which had to be showed. \\

$\bullet$ We deal with the general case now. The morphism $j: \det \F \to \Lambda^rT_X(-D)$ induces a non-zero section $t$ of $ \Lambda^rT_X(-D) \otimes (\det \F)^{-1}$. We choose a log resolution $\nu : \tilde X \to X$ of the ideal sheaf defined by $t$ and $D$, which is co-supported on $W\cup D$. Let us call $D'$ the strict transform of $D$, and analyze the behavior of $\nu^*( \log H)$ along $\Exc(\nu) \cup D'$. By the previous two items, $\nu^*(\log H)$ dominates a (global) finite energy function locally outside $\Exc(\nu) \cap D'$, and its complex Hessian is controlled below by a (negative constant) times a cuspidal metric along $D'$. So let us now see what happens at a point $x \in \Exc(\nu) \cap D'$. Pick a small neighborhood $U$ of $x$ and let $\xi$ be a holomorphic trivialization of $\det \F$ on $U$. Over $U\setminus(W\cup D)$, $\xi$ is a non-vanishing section of $\Lambda^r T_X(-D)$ that can be written $\xi = \sum_{I=(i_1, \ldots, i_r)} a_I \prod_{\substack{i\in I\\ i\le k}}z_i \frac{\d}{\d z_{i_1}}\wedge \cdots \wedge \frac{\d}{\d z_{i_r}} $ as above. Of course the ideal sheaf defined by $t$ is given by $(a_I)$ on $U$. Moreover, as $\om$ is equivalent to a cusp metric along $D$, 
\begin{equation}
\label{comp}
\log |\xi|^2_h = \log \left(\sum_I |a_I|^2 \prod_{i\in I\cap \{1, \ldots, k\}} (-\log |z_i|^2)^{-2}\right)+O(1)
\end{equation}
We write $\nu^*a_I = g_Ig$ where $g$ cuts out the exceptional divisor on $\nu^{-1}(U)$ and $(g_I)$ is a trivial ideal sheaf on this open set. Pick now a point $\tilde x$ above $x$, i.e. $\nu(\tilde x)=x$. There exists $I$ such that $g_I(\tilde x)\neq 0$; moreover, as $\nu^*D$ is an snc divisor, there are holomorphic coordinates $(\tilde z_1, \ldots, \tilde z_n)$ around $\tilde x$ such that each of the $\nu^*z_i$ ($1\le i \le k$) is a monomial in the $\tilde z_j$'s. Therefore, around $\tilde x$, we have: 
$$\nu^*\log |\xi|^2_h \ge \log |g|^2+ \sum_{i\in I\cap \{1, \ldots, k\}} -\log\log^2 |\nu^*z_i|^2-C$$ Take for example $\nu^*z_1$; by the observations above, we have $\nu^*z_1 = \tilde z_1^{a_1} \cdots \tilde z_n^{a_n}$ for some nonnegative integers $a_1, \ldots, a_n$. Then $-\log \log^2 |\nu^*z_i|^2 = -2 \log (-\sum_{i=1}^n a_i \log |\tilde z_i|^2)$ which by concavity of the logarithm is bigger than $\frac{1}{\sum a_i} \sum_{i=1}^n -\log \log^2 |\tilde z_i|^2$. In the end, we obtain:
\begin{equation}
\label{comp2}
\nu^*\log |\xi|^2_h \ge\log |g|^2 + k\sum_{i=1}^n -\log \log^2 |\tilde z_i|^2-C
\end{equation}
Let's gather all what has been said so far. Take $\sigma$ a section cutting out the exceptional divisor $E$ of $\nu$, that we are going to measure with respect to an arbitrary smooth hermitian norm $h_E$. Let us set as before $\vp_P:=\sum -\log \log^2 |s_i|^2$ and we consider the function $$\psi:=\nu^*\left(\log \frac{|\xi|^2_h}{|\xi|^2_{h_0}}+C\vp_P\right)-\log |\sigma|^2$$ on $\nu^{-1}(X\setminus(W\cup D))$. We claim that for $C$ big enough, $\psi$ extends to a quasi-psh function on $\tilde X$ having finite energy. Indeed by the bound $\ddc \log |\xi|^2_h \ge -C \om$, we get that there exists a Kähler metric $\om_0$ on $X$ such that $\ddc (\log |\xi|^2_h+C\vp_P) \ge -\om_0$ outside $W\cup D$, for $C$ big enough. Therefore $\ddc \psi \ge -\tilde \om_0$ on $\nu^{-1}(X\setminus(W\cup D))$ for some Kähler form $\tilde \om_0$ on $\tilde X$.  As $\psi$ is bounded above thanks to \eqref{comp}, it therefore extends as a quasi-psh function on $\tilde X$. Moreover, we retrieve from \eqref{comp2} that near each point of the exceptional divisor (or $D'$), $\psi$ dominates a function satisfying the criterion \cite[Lemma 2.9]{BBGZ} thanks to \cite[Proposition 2.3]{G12}, so $\psi$ satisfies this criterion globally, hence has finite energy. Therefore, 
\[\int_{\nu^{-1}(X\setminus(W\cup D))} \ddc \psi \wedge (\nu^*\om)^{n-1}=0\]
Moreover, $\vp_P$ has finite energy, and $\ddc \log |\sigma|^2= -\Theta_{h_E}(E)$ outside $E$ extends smoothly across across the exceptional divisor (and $D$). As a result, we have
\begin{eqnarray*}
\int_{X\setminus(W\cup D)}\ddc H \wedge \om^{n-1} & = & \int_{\nu^{-1}(X\setminus(W\cup D))} \nu^*\ddc H \wedge (\nu^*\om)^{n-1} \\
&=&\int_{\nu^{-1}(X\setminus(W\cup D))} \ddc[ \psi+\log |\sigma|^2-C\vp_P] \wedge (\nu^*\om)^{n-1}\\
&=&-\int_{\tilde X} \Theta_{h_E}(E) \wedge (\nu^*\om)^{n-1}
\end{eqnarray*}
which vanishes as $E$ is $\nu$-exceptional and $\nu^*\om$ is a finite energy current.
\end{proof}

So we know now that we can use $h$ to compute the slope of $\F$. But outside $W\cup D$, $\F=\Ox(F)$ is a subbundle of $T_X$, so the Chern curvature tensors of each of the hermitian bundles can be related using the second fundamental form $\beta \in \Omega^{1,0} \otimes \mathrm{Hom}(F, F^{\perp})$:
\[\Theta_h(F) = \pr_F( \Theta_h(T_X)_{|F})+\beta^*\wedge \beta\]
where $\pr_F$ is the orthogonal projection onto $F$ of the endomorphism part. Taking the trace (as endomorphism), and then wedging with $\om^{n-1}$, we get:
\begin{equation}
\label{sff}
c_1(F,h) \wedge \om^{n-1} = \tr_{\rm End}(\pr_F(\tr_{\om} \Theta_h(T_X)_{|F})) \, \om^n/n + \tr_{\rm End}(\beta^*\wedge \beta \wedge \om^{n-1} )
\end{equation}
and integrating this identity over $X\setminus(W\cup D)$, we get thanks to Proposition \ref{slope1}:
\begin{equation}
\label{prein}
n \,  c_1(\F) \cdotp  (\pi^ *(K_Y+\Delta)+tA)^{n-1} \le \int_{X\setminus(W\cup D)} \tr_{\rm End}(\pr_F(\tr_{\om} \Theta_h(T_X)_{|F})) \, \om^n 
\end{equation}
The next goal is then to compute the integral in the right hand side and express it in terms of cohomological quantities:
\begin{prop}
\label{tang}
The integral 
\[\int_{X\setminus(W\cup D)} \tr_{\rm End}(\pr_F(\tr_{\om} \Theta_h(T_X)_{|F})) \, \om^n\]
converges to $- r (K_Y+\Delta)^n $ when $\ep$ and then $t$ converge to $0$.
\end{prop}
\noindent
Before starting the proof of the proposition, let us try understand what is $\tr_{\om} \Theta_{h}(T_X)$. We choose geodesic coordinates $(z_i)$ for $\om$ around some point $x_0$, so that
\[\Theta(T_X)_{x_0}=\sum_{j,k,l,m} R_{j\bar k l \bar m} dz_j \wedge d \bar z_k \otimes \left(\frac{\partial}{\partial z_l }\right)^* \otimes \frac{\partial}{\partial \bar z_m } \]
In particular, 
\[\tr_{\om} \Theta(T_X)_{x_0}=\sum_{j,l,m} R_{j\bar j l \bar m} \left(\frac{\partial}{\partial z_l }\right)^* \otimes \frac{\partial}{\partial \bar z_m } \]
and using the Kähler symmetry $R_{j\bar j l \bar m}=R_{l \bar m j\bar j}$, we find 
\[\tr_{\om} \Theta(T_X)_{x_0}=\sum_{j,l,m} R_{l \bar m j\bar j} \left(\frac{\partial}{\partial z_l }\right)^* \otimes \frac{\partial}{\partial \bar z_m } \]
It will be useful to introduce the operator $\sharp$ (relatively to $\om$) which associates to any $(0,1)$-form $\alpha$ a $(1,0)$ vector $\sharp \alpha$ by \[\alpha(\bar u)=\omvp(\sharp \alpha , u)\] for every $u\in T_X$. This operator extends to bundle-valued forms, and one can check easily that if $\alpha$ is a $(1,1)$-form, then $\sharp \alpha$ is an endomorphism of $T_X$ satisfying \[\tr(\sharp \alpha)=\tr_{\om} \alpha.\] If we recall that $\Ric \om = \sum_{j,l,m} R_{l\bar m j \bar j} dz_l \wedge d \bar z_m$, the computation above can also reformulated as 
\begin{equation}
\label{eq:op}
\tr_{\om} \Theta(T_X) = \sharp \Ric 
\end{equation}
We can now resume to the proof:

\begin{proof}[Proof of Proposition \ref{tang} ]
Over $X\setminus D$, our metric $ \om $ satisfies the equation \eqref{eq:ricci}:
\[\Ric \om = -\om +t\om_A-\sum_{a_i>-1} \theta_{i,\ep}\]
therefore $\tr_{\om} \Theta(T_X) = -\sharp \om + t \sharp \om_A - \sum_{a_i>-1} \sharp \theta_{i,\ep}$, and we have three terms to deal with:\\

$\bullet$ Of course, $\sharp \om = \mathrm{Id}_{T_X}$, so that $\tr_{\rm End}(\pr_F(\sharp \om)_{|F})) \, \om^n = r \om^n$, which when integrated over $X\setminus(W\cup D)$ yields $r (\pi^*(K_Y+\Delta)+tA)^n$ as $\om$ has finite energy. So our task is to show that as $\ep$ and $t$ go to zero, the other terms vanish once properly contracted. \\

$\bullet$ As for the second term, $\om_A$ is a positive form, so $\sharp \om_A$ is a positive endomorphism of $T_X$. Hence:
\begin{eqnarray*}
\tr\left(\mathrm{pr}_{F}( (\sharp \om_A)_{|F})\right) \om ^n &\le &\tr(\sharp \om_A) \,\om^n \\
& = & \tr_{\om} \om_A \, \om^n \\
& = & n \om_A \wedge \om^{n-1}
\end{eqnarray*}
and $t \int_{X\setminus(W\cup D)} n\om_A\wedge \om^{n-1}=t n A\cdotp  (\pi^*(K_Y+\Delta)+tA)^{n-1}$ is independent of $\ep$ and converges to $0$ when $t\to 0$. As a result, so does the non-negative integral $\int_{X\setminus(W\cup D)} \tr\left(\mathrm{pr}_{F}( (\sharp \om_A)_{|F})\right) \om ^n$.\\

$\bullet$ The last term is the more subtle to deal with. We are going to show that for each $i$ such that $a_i>-1$, the integral $\int_{X\setminus(W\cup D)}\tr_{\rm End}(\pr_F(\sharp \theta_{i,\ep})_{|F})) \, \om^n $ converges to zero. Remember that $ \theta_{i,\ep} = a_i\left(\frac{ \ep ^2 |D's_i|^2}{(|s_i|^2+\ep^2)^2}+\frac{\ep^2 \Theta_i}{|s_i|^2+\ep^2}\right)$; for the sake of clarity, let us drop the index $i$ and decompose $\theta$ as $\theta = a(\beta+\gamma)$ where $\beta:= \frac{\ep^2 |D's|^2}{(|s|^2+\ep^2)^2}$, and $\gamma =  \frac{\ep^2 \Theta}{|s|^2+\ep^2}$.

\begin{enumerate}
\item[$\cdotp$] Start with $\gamma$. One can choose $C>0$ big enough such that $\pm \Theta \le C \om_A$. As the $\sharp$ operator respects positivity (and hence inequalities), we have 
\begin{eqnarray*}
\left |\tr_{\rm End}(\pr_F(\sharp \gamma)_{|F})) \right |  \om^n &\le& \frac{C \ep^2}{|s|^2+\ep^2} \, \tr\left(\mathrm{pr}_{F}( (\sharp \om_A)_{|F})\right) \om ^n\\
&\le & \frac{nC \ep^2}{|s|^2+\ep^2} \, \om_A \wedge \om^{n-1}
\end{eqnarray*} 
hence 
\begin{equation*}
\left |\int_{X\setminus(W\cup D)} \tr_{\rm End}(\pr_F(\sharp \gamma)_{|F})) \, \om^n \right |  \le C' \int_{X\setminus(W\cup D)} \frac{\ep^2}{|s|^2+\ep^2} \, \om_A \wedge \om^{n-1}
\end{equation*} 
which converges to zero as $\ep$ goes to zero thanks to Lemma \ref{lem2}.\\

\item[$\cdotp$] The term involving $\beta$ cannot be treated in the same way, as it explodes too fast. However, we are going to take advantage of the facts that $\beta$ has a sign and that $\beta+\gamma$ belongs to a fix \textit{contractible} cohomological class. So we have similarly as before:
$$ \tr_{\rm End}(\pr_F(\sharp \beta)_{|F}))   \om^n \le n\beta \wedge \om ^{n-1} $$
Writting $\beta$ as $(\beta+\gamma)-\gamma$, we get:
\[ \int_{X\setminus(W\cup D)} \tr\left(\mathrm{pr}_{F}( (\sharp \beta)_{|F})\right) \om ^n \le C \int_{X\setminus(W\cup D)} \theta \wedge \om^{n-1}+ \frac{\ep^2}{|s|^2+\ep^2} \,  \om_A \wedge \om^{n-1}\]
Using Lemma \ref{lem2} and the fact that $\theta$ is smooth and $\om$ has finite energy, we obtain that the right hand side converges to 
\[C \,E \cdotp (\pi^*(K_Y+\Delta)+tA)^{n-1} \]
when $\ep$ goes to zero, $E$ being the ($\pi$-exceptional) divisor representing the cohomology class of $\theta$. Therefore, our (non-negative) integral converges to $0$ as $t$ goes to $0$, which ends the proof of Proposition \ref{tang}
\end{enumerate}
\end{proof}

\noindent
Combining the inequality \eqref{prein} with Proposition \ref{tang}, we have shown that the slope of $\F$ is less than or equal to the slope of $T_X(-D)$, which concludes the proof of the semistability of $\mathscr T_{Y}(-\Delta)$ with respect to $K_Y+\Delta$.\\

\noindent
\textbf{Step 4. Polystability}\\
Let us prove that the sheaf $\mathscr T_Y(-\Delta)$ is polystable with respect to $K_Y+\Delta$. Using an induction argument, it would be enough to prove that whenever there exists a sheaf $\G \subset \mathscr T_Y(-\Delta)$  with same slope as $\mathscr T_Y(-\Delta)$, then there also exists another sheaf $\G'$ such that  $\mathscr T_Y(-\Delta)\simeq \G\oplus \G'$. Indeed, if $\G$ is chosen of minimal rank amongst the subsheaves of maximal slope, then $\G$ has to be stable, and we can run the whole argument with $\G'$ instead of $\mathscr T_Y(-\Delta)$.

So let us assume the existence of such a subsheaf $\G$. By the arguments of Step 1, we can find a subsheaf $\F$ of $\txdf$ with same slope as $\txdf$, and that we may assume saturated (hence reflexive). 
Let us set $V=\pi^{-1}((Y,\Delta)_{\rm reg})=X\setminus\Supp(E)$, $V_0:=V\setminus D =X\setminus \Supp(D+E)$, and finally $U_0=V_0\setminus W$.\\

Going back to inequality \eqref{sff}; combining it with Proposition \ref{tang} and the fact that $\F$ has same slope as $\txdf$, we obtain that the integral
\begin{equation}
\label{eq:tr}
\int_{U_0} \tr_{\rm End}(\beta_{t,\ep}^*\wedge \beta_{t,\ep}) \wedge \om_{t,\ep}^{n-1}
\end{equation}
converges to zero when $\ep$ and then $t$ go to zero. Here $\beta_{t,\ep}$ is the second fundamental form of $F\subset T_X$ with respect to the hermitian metric induced by $\om_{t,\ep}$.

\noindent
The crucial input we need now is the main Theorem of \cite{BG} which asserts there exists a Kähler metric $\omi$ on $V_0$ such that $\om_{t,\ep}$ converges to $\omi$ in the topology of $\mathscr C^{\infty}$ convergence on the compact subsets of $V_0$. 
Let $\bi$ the second fundamental form of $F\subset T_{X}$ induced by the Kähler metric $\omi$ on $U_0$. By the smooth convergence of $\om_{t,\ep}$ to $\omi$ on the compacts of $V_0$, Fatou's lemma applied to \eqref{eq:tr} shows that: 
\[ \int_{U_0} \tr( \bi^* \wedge \bi) \wedge \omi^{n-1} =0 \]
so that $\tr( \bi^* \wedge \bi)$ and hence $\bi$ vanishes on $U_0$. In particular, we get an holomorphic splitting \begin{equation}
\label{split}
\Txdf_{|U_0} \simeq \F_{|U_0} \oplus \F_{|U_0}^{\perp}
\end{equation}
We would like to extend this decomposition to $V=\pi^{-1}((Y, \Delta)_{\rm reg})$ and push it forwards to $(Y, \Delta)_{\rm reg}$ where it would extend to $Y$ by reflexivity. The difficulty here is that $U_0\subset V$  has codimension $1$ because of $D$. So we have to show first how extend the splitting across $D$ and then use reflexivity arguments to get it on $V$. \\

Another way to reformulate \eqref{split} is to say that there is an holomorphic surjection of vector bundles $p:\Txdf \longrightarrow \F$ over $U_0$. Or better, we can view it as morphism $p:\Txdf \longrightarrow \Txdf$ over $U_0$ satisfying $p^2=p$, and whose norm is less that $1$, ie for every $v\in \Txdf$, we have  $|p(v)|_{\omi} \le |v|_{\omi}$. We want to extend $p$ to $U\cap V$. The additional information we need lies in \cite[Theorem B]{GW} where it is shown that $\omi$ is equivalent to a cusp metric along $D$ on $V$. Now we choose a point $x\in D\cap V$, such that $D$ is locally given by $D=(z_1\cdots z_k=0)$ for some coordinates $(z_{i})$ around $x$. Then 
\[e_i:=\begin{cases}
z_i \frac{\d}{\d z_i} &  \text{if $1 \le i \le k$} \\
\frac{\d}{\d z_i} & \text{if $ i > k$} 
\end{cases}
\]
defines an holomorphic frame of $\Txdf$ around $x$. By the previously cited result, the $\omi$-norm of $e_i$ is equivalent to $(\log \frac{1}{|z_i|^{2}})^{-1}$ if $1 \le i \le k$, and $1$ if $i>k$. Let us  consider $(p_{ij}(z))$  the matrix of $p$ with respect to the basis $(e_i)$.
We now that there exists a constant $C$ such that $|p(e_i)|_{\omi} \le C$ for all $i$. One the other hand, as $\omi$ is equivalent to the cusp metric, we have
\begin{eqnarray*}
|p(e_i)|_{\omi}^2&=&\left|\sum_{j=1}^n p_{ij}(z) e_j \right|_{\omi}^2 \\
& \ge & C^{-1} \left( \sum_{j=1}^r  \frac{|p_{ij}(z)|^2}{\log^{2} \frac{1}{|z_i|^{2}}} + \sum_{j=r+1}^n |p_{ij}(z)|^2 \right)
\end{eqnarray*}
Therefore, $|p_{ij}(z)|$ is uniformly bounded near $x$ if $j>k$, and $|p_{ij}(z)| \le \log^{2} \frac{1}{|z_i|^{2}}$ if $1\le j \le k$. In any case, $p_{ij}$ is locally $L^2$ near $x$ (with respect to a smooth volume form), and therefore it extends analytically across $D$. As a consequence, we get a holomorphic splitting
\[\Txdf_{|U\cap V} \simeq \F_{|U\cap V} \oplus \F_{|U\cap V}^{\perp}\]
As the complement of $U\cap V$ in $V$ has codimension at least $2$ and all the sheaves at play are reflexive, the above decomposition extends as a sheaf isomorphism
\[\txdf_{|V} \simeq \F_{|V} \oplus (j_{U\cap V})_*(\F_{|U\cap V}^{\perp})\]
where $j_{U\cap V}:U\cap V \to V$ is the open immersion. Pushing forward by $\pi$, we get on $(Y, \Delta)_{\rm reg}$:
\[\mathscr{T}_{Y}(-\Delta)_{|(Y, \Delta)_{\rm reg}}\simeq \G_{|(Y, \Delta)_{\rm reg}}\oplus \G'_{|(Y, \Delta)_{\rm reg}} \]
where $\G'$ is the (reflexive) sheaf on $Y$ defined as $(j_{(Y, \Delta)_{\rm reg}} \circ \pi \circ j_{U\cap V})_*(\F_{|U\cap V}^{\perp})$ where $j_{|\cdotp}:\, \cdotp \to Y$ generically denotes any open immersion from $\cdotp$ to $Y$. And here again, the reflexivity of these sheaves leads to the expected splitting: 
\[\mathscr{T}_{Y}(-\Delta)\simeq \G\oplus \G' \]
on $Y$, ending the proof of the polystability of $\mathscr{T}_{Y}(-\Delta)$.

\end{proof}

\begin{rema}
\label{remarque}
We already outlined in the introduction and in \S \ref{statement} that Theorem \ref{thm:pair} is a significant generalization of what was known before, i.e. the semistability of the tangent sheaf in the case of canonical singularities, cf \cite{Enoki}. However, two consequential simplifications occur if one is interested in a weakened form of Theorem \ref{thm:pair}:

$\cdotp$ First, if one only focuses on semistability and not polystability, then instead of using the inequality relating the Chern curvature form of a subbundle $F$ to the one of the ambient bundle $T_X$ (cf \eqref{sff}), one can use a weaker form of it as in \cite[equation (1.2)]{CP2} which enables to endow $ \det F$ with a smooth metric instead of the singular metric induced by $\om$ (the singularities occurring because $F$ is a subbundle only on an open set, and because $\om$ may have cuspidal singularities along a divisor), thus avoiding the use of the technical Lemma \ref{slope1}. 

$\cdotp$ Second, if $\Delta=0$, then the approximate Kähler-Einstein metrics $\om$ are smooth and don't have cuspidal singularities anymore. In that case, the crucial Lemma \ref{slope1} becomes much easier and was already treated in the item $(\ast\ast)$ of the proof of \cite[Theorem 8.3]{Koba}.
\end{rema}

\section{An extension to stable varieties}
\label{sec:stable}

\subsection{Stable varieties}
\label{sec:stable0}

Once we have studied log canonical pairs with ample log canonical bundle, it is very tempting to try to extend these results to varieties (or pairs) with possible non-normal singularities. More precisely, we will consider \textit{stable varieties} in the sense of Koll\'ar-Shepherd-Barron \cite{KSB} and Alexeev \cite{Alex}. These varieties are the higher dimensional analogue of stable curves as defined by Deligne-Mumford \cite{DM} and they arise in the compactification of the moduli space of smooth canonically polarized projective varieties; their precise definition stated below, though we refer to the very nice paper \cite{Kov} and the references therein for more details and insight on all the objects involved.

\begin{defi}
A stable variety is an equidimensional and reduced complex projective variety $X$ with semi-log canonical singularities such that $K_X$ is ample.
\end{defi}

Let us recall that $X$ has semi-log canonical singularities if the only singularities of $X$ in codimension $1$ are double points (i.e. locally analytically isomorphic to $(xy=0)\in \CC^{n+1}$), $X$ satisfies Serre's $S_2$ property, is $\Q$-Gorenstein, and the pair $(\Xnu, \Dnu)$ formed by the normalization of $X$ and its conductor divisor has log canonical singularities. Here the conductor of the normalization is essentially the Weil divisor determined by the preimage of the double points locus by the normalization map. 

It is important to notice that stable varieties may be reducible, which is a source of difficulties in view of a theory of semistability for coherent sheaves. Indeed, if we want to mimic the definition of the slope of a coherent sheaf $\F$, we need to define the $(n-1)$-cycle $c_1(\F)$. We can still do it on $\xreg$, but there is an issue when we want to extend it to the whole $X$ as the complement of $\xreg$ has codimension one in general. Actually, there is a very general notion of slope, defined in a high degree of generality using Hilbert polynomials (see e.g. \cite{HL}), and of course agreeing with the usual one on smooth varieties. We are now going to review these notions briefly.

\subsection{Slope and stability on non-normal varieties}

On non-normal varieties, defining the slope of a coherent sheaf is more complicated. We recall the definition, cf \cite{HL}. Let $X$ be a $n$-equidimensional reduced projective scheme over $\CC$, $L$ an ample line bundle, and $\F$ a coherent sheaf on $X$ of dimension $n$ (which means that the support of $\F$ has dimension $n$). We know that the Hilbert polynomial $P(\F)$ given by $m \mapsto \chi(X,\F \otimes L^{\otimes m})$ can be uniquely written in the form: 
\[P(\F,m):= \sum_{k=0}^n a_k(\F) \frac{m^k}{k!}\]
where $a_k(\F)$ are integers. Note that for $m$ large enough, the vanishing of the higher cohomology implies that $P(\F,m)= h^0(X,\F\otimes L^{\otimes m})$. Of course, this polynomial depends on $\F$ and on the polarization $L$.
If $X$ is integral, then the rank of $\F$ is defined as the rank of $\F$ at the generic point (as $\F$ is locally free on some dense open set). In this more general framework, we define the rank of $\F$ as a convex combination of the rank of $\F$ restricted to each irreducible components. More precisely, if $X=X_1 \cup \ldots \cup X_r$ are the irreducible components of $X$ and if $L_i$ (resp. $\F_i$) denotes the restriction of $L$ (resp. $\F$) to $X_i$, then we define
\[\rk(\F):= \frac{\sum_{i=1}^r (L_i^n) \,  \rk( \F_i )}{(L^n)}\]
One can also check that $\rk(\F)$ coincides with $a_n(\F)/a_n(\Ox)$, cf \cite[Remark 1.1.26]{Laz1}. If $\F$ has constant rank (ie $\rk(\F_i)$ is independent of $i$), then the rank of $\F$ is the same as $\rk(\F_i)$ for any $i$. 

Now we need to define the degree of $\F$ with respect to $L$. To do it, we introduce the suitable combination of the coefficients $a_k$ that gives in the smooth case the usual degree. More precisely:
\[\deg_L(\F):= a_{n-1}(\F)-\rk(\F)\, a_{n-1}(\Ox)\]
and finally, one can define the slope as the quotient of the degree by the rank:
\[\mu_L(\F):=\frac{\deg_L(\F)}{\rk(\F)}\]
Of course this definition only make sense for sheaves with positive rank.
\begin{rema}
\label{rem}
If $X=X_1\amalg \cdots \amalg X_s$ is the disjoint sum of its irreducible components, and if $\F$ has constant rank (ie $\rk(\F_{|X_i})$ does not depend on $i$), then it follows from the definition that the slope of $\F$ with respect to a polarization $L$ on $X$ is the sum of the slopes of $\F_{|X_i}$ with respect to $L_{|X_i}$
\end{rema}

The first thing to check is whether this definition generalizes the the definition of slope that we gave for normal varieties. Recall that if $\F$ is a coherent sheaf on a normal variety we defined $c_1(\F)$ to be the $(n-1)$-cycle represented by the closure of the Weil divisor attached to $\det (\F_{|\xreg})$.

\begin{prop}
Let $(X,L)$ be a polarized normal variety of dimension $n$, and $\F$ a torsion-free coherent sheaf on $X$. Then we have
\[\deg_L(\F)= (c_1(\F) \cdotp L^{n-1})\]
and in particular, $\mu_L(\F)$ agrees with $(c_1(\F) \cdotp L^{n-1})/\rk(\F)$.
\end{prop}

\begin{proof}
Let $\pi:Y\to X$ a resolution of $X$ such that $\pi$ is an isomorphism over $\xreg$. We will show that these two quantities above coincide with $(c_1(\pi^*\F) \cdotp \pi^*L^{n-1})$. 

Using the projection formula and the fact that on $\xreg$, the cycles $\pi_* c_1(\pi^*(\F))$ and $c_1(\F)$ coincide, we see that $(c_1(\pi^*\F) \cdotp \pi^*L^{n-1})=(c_1(\F) \cdotp L^{n-1})$, so we did half of the job. The other equality requires a bit more work.

First, let us prove that $\deg_L(\F)= \deg_{\pi^*L}(\pi^*\F)$. If we define $\G$ to be the sheaf $\pi^* \F \otimes \pi^*L^{\otimes m}$, then the degeneracy of Leray spectral sequence provides the equality
\begin{eqnarray}
\label{spectral}
\chi(Y, \G)= \sum_{q\geq 0}(-1)^q \chi(X,R^q \pi_*\G)
\end{eqnarray}
Moreover, the projection formula gives us a  isomorphism
\[  R^q \pi_*(\pi^*\F\otimes_{\O_Y} \pi^*L^{\otimes m}) \longrightarrow  R^q \pi_*(\pi^*\F)\otimes_{\Ox} L^{\otimes m} \]
As $\F$ is torsion-free, it is locally free in codimension one (cf e.g. \cite[5.15]{Koba}), so if we apply the projection formula once again, we obtain that the natural morphism \
\[R^q \pi_*(\pi^*\F)  \longrightarrow R^q \pi_*{\mathcal O_Y} \otimes_{\Ox} \F \]
is an isomorphism in codimension two. Combining these two observations, we get a morphism
\[  R^q \pi_*(\pi^*\F\otimes_{\O_Y} \pi^*L^{\otimes m}) \longrightarrow R^q \pi_*{\mathcal O_Y} \otimes_{\Ox} \F \otimes_{\Ox} L^{\otimes m}  \]
which is an isomorphism in codimension two. Therefore, as $\pi_*\O_Y = \O_X$, we get from \eqref{spectral}:

\[\chi(Y, \pi^* \F \otimes \pi^*L^{\otimes m})= \chi(X,\F\otimes L^{\otimes m})+\sum_{q> 0}(-1)^q \chi(X,R^q \pi_*\O_Y\otimes \F \otimes L^{\otimes m})+O(m^{n-2}) \]
Finally, as $X$ is normal and $\pi$ is an isomorphism over $\xreg$, the sheaves $R^q \pi_* \O_Y$ are supported in codimension $\leq 2$ for $q>0$. (cf \cite[8.2 \& 11.2]{Har} for example).
Therefore, we obtain:
\[\chi(Y, \pi^* \F \otimes \pi^*L^{\otimes m})= \chi(X,\F\otimes L^{\otimes m})+O(m^{n-2}) \]
These two polynomials have same leading term coefficient, namely the top intersection of $c_1(\F)\otimes L$ divided by $n!$ (as $\pi$ is generically of degree one), so they have the same coefficient in front of the second highest power of $m$.
As this is true for any torsion-free sheaf $\F$, we have the same result for $\O_X$ too, so in the end we get $\deg_L(\F)= \deg_{\pi^*L}(\pi^*\F)$. 

To conclude the proof of the proposition, we need to see that $\deg_{\pi^*L}(\pi^*\F)=(c_1(\pi^*\F) \cdotp \pi^*L^{n-1})$. By Riemann-Roch-Hirzebruch theorem, we know that
\begin{eqnarray*}
\chi(Y, \pi^* \F \otimes \pi^*L^{\otimes m}) &=&
 \rk(\F) \cdotp (\pi^*L)^n \frac{m^n}{n!}+\Big[(c_1(\pi^*\F) \cdotp \pi^*L^{n-1})+\\
&& +\frac{1}{2}(c_1(Y)\cdotp \pi^*L^{n-1})\Big] \frac{m^{n-1}}{(n-1)!}+O(m^{n-2})
 \end{eqnarray*}
Applying this to $\F$ and $\O_Y$, we get the expected equality.
\end{proof}

Now we would like to see, for non-normal variety, how to relate the slope of a sheaf with the slope of its pull-back by the normalization map. This is the content of the following proposition:

\begin{prop}
\label{prop:norm}
Let $X$ be a reduced equidimensional projective scheme over $\CC$, $L$ a polarization, and $\F$ a coherent sheaf which is locally free in codimension one. If $\pi:Y\to X$ is the normalization of $X$, then we have:
\[\mu_{\pi^*L}(\pi^*\F) = \mu_L(\F)\]
\end{prop}

\begin{proof}
First, let us define $\T:=\pi_*\O_Y/\Ox$, and let set $r=\rk(\F)$ (one may notice that $\F$ being locally free in codimension one, it has constant rank). As $L$ and $\pi^*L$ are ample, we will assume in the following that the power $m$ to which we raise them is large enough so that all the higher cohomology groups involved are zero. As a consequence, we will identify the Euler-Poincaré characteristic with $h^0$.
Using the projection formula combined with a similar argument as in the previous proof  (we use here that $\F$ is locally free in codimension one) we see that 
\[h^0(Y, \pi^*\F \otimes \pi^*\Lm )=h^0(X,\F \otimes \Lm \otimes \pi_*\O_Y)+O(m^{n-2})\]
Moreover, as $\F$ is locally free in codimension one, the short sequence
\[0\to \F \otimes \Lm \to \F \otimes \Lm \otimes \pi_*\O_Y \to \F \otimes \Lm \otimes \T \to 0\]
is exact in codimension one. Combining these two identities, we get: 
\[a_{n-1}(\pi^*\F)= a_{n-1}(\F)+a_{n-1}(\F\otimes \T)\]
where the polarizations which respect to which the Hilbert polynomials are computed are respectively $L$ on $X$ and $\pi^*L$ on $Y$.

Therefore, we have:
\begin{eqnarray*}
\mu_{\pi^*L}(\pi^*\F) & = & \frac{1}{r} a_{n-1}(\pi^*\F)-a_{n-1}(\O_Y) \\
&=& \frac{1}{r} a_{n-1}(\F)+\frac{1}{r}a_{n-1}(\F \otimes \T)-a_{n-1}(\O_X)+(a_{n-1}(\O_Y)-a_{n-1}(\O_X))\\
&=&\mu_L(\F)+\frac{1}{r}a_{n-1}(\F\otimes \T)-a_{n-1}(\T)
\end{eqnarray*}
If $\T$ is supported in codimension $\leq 2$, then $\frac{1}{r}a_{n-1}(\F\otimes \T)=a_{n-1}(\T)=0$ and we are done. Else, it is supported in a codimension one subvariety $Z$, so it follows from Riemann-Roch formula (in the reducible case, cf \cite[Proposition VI.2.7]{Kollar96}) that $a_{n-1}(\F \otimes \T)= \rk((\F\otimes \T)_{|Z})=r \cdotp \rk(\T) = ra_{n-1}(\T)$, where we used again that $\F$ is locally free in codimension one. In the end, we have proved that the slopes of $\F$ and $\pi^*\F$ agree.
\end{proof}

\begin{exem}
One should pay attention to the fact the the above result is false without the assumption that the sheaf is locally free in codimension one. Indeed, if we choose $X=(xy=0) \subset \P^2$ to be the union of two lines, then its normalization is $\P^1$ with two points $p',p''$ sitting above the node $p$. Then, if we choose $L=\O_{\P^2}(1)_{|X}$ and $\F= \Ox \oplus (\Ox/\mathcal I_p)$ where $I$ is the ideal sheaf of the node, we see that $h^0(\F\otimes \Lm)=h^0(\Lm)+1$, ie $\mu_L(\F)=1$. But $\pi^*\F=\O_Y\oplus \O_Y/\mathcal I_{p',p''}$ so that $h^0(\pi^*\F\otimes \pi^*\Lm)=h^0(\pi^*\Lm)+2$ and therefore $\mu_{\pi^*L}(\pi^*\F)=2$.
\end{exem}

In view of all the observations we made so far, one can introduce a notion of semi-stability which is weaker than the one in \cite[1.2.12]{HL}, but particularly adapted to our context:

\begin{defi}
Let $X$ be a projective $n$-equidimensional reduced projective scheme over $\CC$, and $L$ be an ample divisor. A coherent sheaf $\E$ is said semistable with respect to $L$ if for every subsheaf $\F\subset \E$ being locally free in codimension one and satisfying $0<\rk(\F)<\rk(\E)$, we have $\mu_L(\F) \le \mu_L(\E)$.
\end{defi}

If $X$ is normal and $\E$ is torsion-free, then any subsheaf $\F$ of $\E$ is also torsion-free hence locally free in codimension $1$ as $X$ is normal. So we just recover the usual notion of semistability. 

\subsection{The result}
The general definition above is actually motivated by Proposition \ref{prop:norm} and tailored for the following Theorem:

\begin{theo}
\label{thm:stable}
Let $X$ be a stable variety, let $\nu:X^{\nu}\to X$ be the normalization of $X$ and $\Delta$ be the conductor of $\nu$. Then the sheaf $\nu_*\mathscr T_{X^{\nu}}(-\Delta)$ is semistable with respect to $K_X$.
\end{theo}

\begin{proof}
Let $\F\subset \Tx$ be a coherent subsheaf of $\Tx$ which is locally free in codimension one. If $\nu:X^{\nu}\to X$ is the normalization of $X$ and $\Delta$ be the conductor divisor of $\nu$ on $X^{\nu}$, then we get a generically injective map $\nu^*\F \to \nu^*\nu_*\mathscr T_{\Xnu}(-\Delta)$. If we compose if with the natural map $\nu^*\Tx \to (\nu^*\nu_*\mathscr T_{\Xnu}(-\Delta))^{**}$, then in view of the lemma \ref{lem:norm} below, we get a generically injective morphism
\begin{equation*}
\label{eq:nor}
\nu^*\F \to \mathscr T_{X^{\nu}}(-\Delta)
\end{equation*}
Let us now consider $\pi:Y\to X^{\nu}$ a log resolution of $(X^{\nu}, \Delta)$ that leaves the snc locus untouched, and let $D=\pi^{-1}(\Delta)$.
Then there is a natural morphism $\pi^*T_{X^{\nu}}(-\Delta) \to \mathscr T_Y(-D)$ which is an isomorphism over $(X^{\nu}, \Delta)_{\rm reg}$. Taking to the double dual and using the reflexivity of $\mathscr T_Y(-D)$, we end up with a natural isomorphism $(\pi^*T_{X^{\nu}}(-\Delta))^{**} \simeq \mathscr T_Y(-D)$. Pulling back \eqref{eq:nor} by $\pi$, we get a generically injective morphism $\pi^* \nu^* \F \to \mathscr T_Y(-D)$. We proved that in this situation, the slope of $\pi^*\nu^* \F$ with respect to $\pi^*\nu^*K_X$ is less or equal to the slope of $\mathscr T_Y(-D)$. At this point, we use the fact that $\F$ has constant rank to pass from the slope inequality on each component to the global slope inequality.

 By the projection formula, it yields
\[\mu_{\nu^*K_X}(\nu^*\F) \le \mu_{\nu^*K_X}(T_{X^{\nu}}(-\Delta))\]
We would like to push forward to $X$ this inequality. For the left hand side, this can be done thanks to Proposition \ref{prop:norm} which shows that $\mu_{\nu^*K_X}(\nu^*\F) = \mu_{K_X}(\F)$ as $\F$ is locally free in codimension one. Actually, the right hand side is not equal to $\mu_{K_X}(\Tx)$, but there is a correction factor which has the right sign fortunately. 
More precisely, writing $\T:=\nu_*\O_{X^{\nu}}/\Ox$, we have by the projection formula: 
\[h^0(\Xnu, \nu^{*}K_X^{\otimes m})= h^0(X,\Km \otimes \nu_*\O_{\Xnu})\]
and also
\[h^0(\Xnu, \mathscr T_{\Xnu}(-\Delta)\otimes \nu^{*}K_X^{\otimes m})= h^0(X,\nu_*\mathscr T_{\Xnu}(-\Delta) \otimes \Km )\]
Therefore, we get
\begin{eqnarray*}
\mu_{\nu^*K_X}(T_{X^{\nu}}(-\Delta)) & = & \frac{1}{n}a_{n-1}(T_{X^{\nu}}(-\Delta))-a_{n-1}(\O_{\Xnu})\\
&=& \left(\frac{1}{n}a_{n-1}(\Tx)-a_{n-1}(\Ox) \right) + a_{n-1}(\Ox) -a_{n-1}(\O_{\Xnu}) \\
&=& \mu_{K_X}(\nu_*T_{\Xnu}(-\Delta)) - a_{n-1}(\T)
\end{eqnarray*}
As $\T$ is supported in codimension $\ge 1$, we have $a_{n-1}(\T) \ge 0$, so that in the end, we have proved
\[ \mu_{K_X}(\F) \le \mu_{K_X}(\nu_*\mathscr T_{\Xnu}(-\Delta)) \]
\end{proof}

\begin{rema}
\label{rm:ct}
In general, the sheaf $\nu_*\mathscr T_{X^{\nu}}(-\Delta)$ is distinct from the tangent sheaf of $X$ (defined as the sheaf of derivations of $\Ox$) as already in the case of a (local) node $X= \mathrm {Spec}(\mathbb C[x,y,z]/(xy))$ the former sheaf strictly contains the latter one. But given the last step of the proof where we neglected $a_{n-1}(T)$, it could be possible that the same proof yields the semistability of $\mathscr T_X$.
\end{rema}

To conclude this section, let us state and prove a result we have used in the course of the proof of Theorem \ref{thm:stable}:

\begin{lemm}
\label{lem:norm}
Let $X$ be a projective $n$-equidimensional reduced projective scheme over $\CC$ and let $\pi:Y\to X$ the normalization of $X$. If $\E$ is a reflexive coherent sheaf on $Y$, then the reflexive hull $(\pi^* \pi_* \E)^{**}$ is naturally isomorphic to $\E$.
\end{lemm}

\begin{proof}
Let us begin by reducing the problem to one of commutative algebra. First, there is a natural map $\pi^*\pi_*\E \to \E$ inducing $(\pi^*\pi_*\E)^{**} \to \E^{**}\simeq \E$. This map is a morphism between two reflexive sheaves, so if we can show that it is a isomorphism over a open set in $Y$ whose complement has codimension at least two, then our map will actually be a sheaf isomorphism over the whole $Y$. As a consequence, one can assume that $\E$ is locally free (as $\E$ is reflexive, it will be locally free on such a "big" open set). Now, the problem is local, and as the normalization is an affine morphism, one can assume without loss of generality that $Y=\mathrm{Spec}(B)$, $X=\mathrm{Spec}(A)$, and that $\E=\widetilde B$. We claim that the sheaf morphism $\E^* \to (\pi^*\pi_*\E) ^*$ is an isomorphism, or equivalently that
the following morphism of $B$-modules:
\begin{align*}
  F \colon \Hom(B,B) &\longrightarrow \Hom(B\otimes_A B_A, B)\\
  \vp &\longmapsto \vp \circ f
\end{align*}
is an isomorphism. Here, $B_A$ is $B$ viewed as an $A$-module and $f:B\otimes_A B_A \to B$ is defined by $f(\sum b_i \otimes m_i)=\sum b_im_i$. As $f$ is surjective, $F$ is injective, so we need to prove that $F$ is surjective. Given $\psi \in \Hom(B\otimes B_A, B)$, we define $\vp \in \Hom(B,B)$ by setting $\vp(1)= \psi(1\otimes 1)$. We want to show that $F(\vp)=\psi$, and this amounts to proving that for any $b\in B$, we have $\psi(1\otimes b)=b\psi(1\otimes 1)$. The crucial input is that $B$ is a subring of the total quotient ring of $A$, so there exists $s\in A$ which is not zero divisor such that $sb\in A$. Therefore, if we set $x:=\psi(1\otimes b)-b\psi(1\otimes 1)$, we get $sx= s\psi(1\otimes b)-sb\psi(1\otimes 1)= \psi(s\otimes b) - \psi(sb\otimes 1)=\psi(1\otimes sb)-\psi (sb\otimes 1)=0$
as $sb\in A$. So $sx=0$; but $s$ is not a zero divisor in $A$ and hence in $B$ too as $B$ is a subring of the total quotient ring of $A$. This proves that $x=0$, and thus $F$ is an isomorphism as claimed.
\end{proof}

\section{More stability and generic semipositivity}
\label{sec:gsp}

In this last section, we try to loosen the positivity assumptions on $K_X$ (replacing ample by nef, or even anti-nef). In view of Miyaoka's semipositivity theorem (cf discussion after Theorem \ref{thm:gsp}), it is natural to expect that the cotangent sheaf (resp. tangent sheaf) is generically semipositive with respect to any Kähler class. This is what we are going to show by adapting the methods of the previous sections to this different context. 

But before that, let us recall that on a complex normal variety $X$, the reflexive sheaf of differentials, that we will denote by $\Omega_X^1$ is defined to be the pushforward $j_{*}\Omega_{\xreg}^1$ by the open immersion $j: \xreg \hookrightarrow X$ of the regular locus into $X$. Equivalently, $\Omega_X^1$ can be defined as the reflexive envelope of the push-forward of the bundle of differentials $(\pi_* \Omega_{X'})^{**}$ for any resolution $\pi:X'\to X$. The sheaf $\pi_* \Omega_{X'}$ is already reflexive in the case of log terminal singularities, but it may not the case anymore if one merely assume that the singularities are log canonical (cf. \cite[Theorem 1.4 \& 1.5]{GKKP}). 

The main theorem of this section is the following:

\begin{theo}
\phantomsection
\label{thm:gsp}
Let $X$ be a $n$-dimensional compact Kähler space with log canonical singularities, and $\om$ a Kähler form. If $K_X$ is nef (resp. $-K_X$ is nef), then $\Omega_X^1$ (resp $\Tx$) is generically $\om$-semipositive. 
\end{theo}

Applying these results to singular Calabi-Yau varieties, we can even show:

\begin{coro}
\label{cy}
Let $X$ be a singular Calabi-Yau variety, i.e. $X$ is a complex projective variety with at most klt singularities and such that $K_X$ is numerically trivial. Then $\mathscr T_X$ is polystable with respect to any polarization.
\end{coro}

Recall also that a reflexive sheaf $\E$ on a normal compact Kähler space $X$ endowed with a Kähler form $\om$ of is said to be generically $\om$-semipositive if for all coherent quotient $\F$ of $\E$, the degree of $\F$ with respect to $\om$ is non-negative, ie 
$\int_X c_1(\F) \wedge \om^{n-1} \ge 0$. 
By a theorem of Hartshorne \cite{Har3}, a vector bundle over a curve is nef if and only if all its quotient bundles have non-negative degree. So as a consequence, over a projective variety the restriction of a generically semipositive sheaf/bundle to a sufficiently general curve is nef. \\

The case $K_X$ nef of Theorem \ref{thm:gsp} is a weak form of Miyaoka semiposivity theorem \cite{Miyaoka}, as this celebrated theorem holds for every normal projective variety, any polarization $(H_1, \ldots, H_{n-1})$ and only with the assumption that $X$ is not uniruled, which is automatic if $-K_X$ is nef and $X$ is smooth (or merely has canonical singularities). However, our result covers some new cases for the following two reasons:

$\cdotp$ It holds for (singular) Kähler spaces and not only for algebraic varieties \--- it was also proved in the smooth Kähler case by Junyan Cao \cite{Cao}.

$\cdotp$ As there exist rational (hence uniruled) surfaces with only klt singularities and for which $K_X$ is ample \cite[Example 43]{Ko08}, our assumption that $K_X$ is nef encompasses some new cases.
 
This result is closely related with two recent results. First, F. Campana and M. P\u{a}un \cite{CP, CP2} showed that given a log smooth pair $(X,D)$ such that $K_X+D$ is pseudo-effective, then one can dominate the slope of any subsheaf of $\otimes^m \Omega_X(\log D)$ by $m$ times the intersection $c_1(K_X+D) \cdotp \{\om\}^{n-1}$ if $\om$ is the Kähler form with respect to which the slope is computed. Also, D.Greb, S.Kebekus and T. Peternell  obtain in \cite[Theorem 1.3]{GKP} a structure theorem for the tangent sheaf of klt varieties with trivial canonical bundle (and canonical singularities): it decomposes into a sum of stable sheaves after a finite cover, étale in codimension one. Let us finally mention that it is conjectured \cite[Conj 1.3]{Peternell} that a projective manifold with nef anticanonical bundle has a generically nef tangent bundle (ie semipositive with respect to every polarization $(H_1, \ldots, H_{n-1})$).

\begin{proof}[Proof of Theorem \ref{thm:gsp}]
Let us begin with the case $K_X$ nef. Here again, as in the previous proofs, we denote by $Y$ the singular original Kähler space, and choose $\pi : X\to Y$ a resolution. We will show that for each coherent subsheaf $\F \subset \mathscr T_X$ and every Kähler form $\om_Y$ on $Y$, we have $\int_X c_1(\F) \wedge (\pi^*\om_Y)^{n-1} \le 0$. This will as before show that each coherent subsheaf of $\mathscr T_Y$ has non-positive slope with respect to $\om$, and this will conclude by duality (consider a quotient $\mathcal G$ of $\Omega_Y^1$, it induces a subsheaf $\mathcal G^* \subset \mathscr T_Y$ with non-positive slope, so that $\mathcal G$ has a non-negative slope). \\

We write $K_X= \pi^*K_Y+D$ where $D=\sum_{i\in I} a_i D_i$ is a $\pi$-exceptional $\Q$-divisor with snc support and coefficients $a_i\ge -1$. We choose sections $s_i$ cutting out the components $D_i$, and fix smooth hermitian metrics $|\cdotp|_i$ whose Chern curvature form will be denoted by $\Theta_i$.

Also, we know that $\pi^*K_Y$ is nef, so for every $t>0$, there exists a smooth form $\om_t\in c_1(\pi^*K_Y)$ such that $\om_t \ge - t \om_A$, where $\om_A$ is a fixed Kähler form on $X$; we will write $A:= \{\om_A\}$ for its cohomology class in $H^{1,1}(X,\R)$. 
Let $\om_Y$ a Kähler form on $Y$, and set $H:=\{\om_Y\}$ its cohomology class; we introduce the pull-back of $\om_Y$ by the resolution: $\om_X= \pi^*\om_Y$; it is a semipositive big form and for each $t>0$, $\om_X+t\om_A$ is a Kähler class.
We introduce a regularizing parameter $\ep>0$ as before, but here we will need a new parameter $\delta>0$ to deal with the log canonical singularities without introducing cuspidal metrics (it turns out that this is possible since $D$ is exceptional, but we couldn't have used this trick to show the semistability of the logarithmic tangent sheaf earlier). 

Thanks to Yau's theorem, we can find a smooth solution $\vp$ of the following Monge-Ampère equation:  
\begin{equation}
\label{eq:ke2n}
(\om_X+ t\om_A+\ddc \vp)^n = \prod_{i \in I} (|s_i|^2+\ep^2)^{(1-\delta)a_i}dV_t
\end{equation}
where $dV_t$ is a volume form whose Ricci curvature is $\Ric dV_t = -\om_t-\sum a_i \Theta_i$. Setting as before $\om:=\om_X+t\om_A+\ddc \vp$ (so $\om$ depends on $\ep, t$ and $\delta$), we find:
\begin{equation}
\label{eq:ric4n}
\Ric \om = -\om_t -(1-\delta)\theta-\delta \Theta
\end{equation}
where we set
$\theta =  \sum_{i\in I}a_i\left(\frac{ \ep ^2 |D's_i|^2}{(|s_i|^2+\ep^2)^2}+\frac{\ep^2 \Theta_i}{|s|^2+\ep^2}\right) $, and $\Theta=\sum_{i\in I}a_i \Theta_i$.\\

The Kähler form $\om$ induces a genuine hermitian metric $h$ on $T_X$ and therefore also on the subbundle $F$
over $X\setminus W$. The first bullet in the proof of Proposition \ref{slope1} shows that the degree of $\F$ (against $\{\om\}$ can be computed by integrating $c_1(F,h)\wedge \om^{n-1}$ over $X\setminus W$. And by the inequality relating the Chern curvature of a subbundle to the one of the ambient bundle, we find as in \eqref{prein} that 
\[c_1(\F) \cdotp \{\om\}^{n-1} \le \int_{X\setminus W} \tr_{\rm End} \, (\pr_F (\sharp \Ric \om)_{|F}) \, \om^{n}\]

Choosing a constant $C$ such that $\pm \Theta \le C \om_A$, and remembering that $-\om_t \le t \om_A$, we get:
\[c_1(\F) \cdotp \{\om\}^{n-1} \le \int_{X\setminus W}n(t+C\delta) \om_A\wedge \om^{n-1}+ (1-\delta)\tr_{\rm End} \, (\pr_F (\sharp \theta)_{|F}) \, \om^{n}\]
Using the same arguments as in the previous proof, we obtain
\[c_1(\F) \cdotp \{\om\}^{n-1} \le \left[n(t+C\delta) A+ (1-\delta) D\right] \cdotp (\pi^*H+tA)^{n-1}+ C\sum_i \int_X \frac{\ep^2}{|s_i|^2+\ep^2} \om_A \wedge \om^{n-1}   \]
By an easy adaptation of Lemma \ref{lem2} to our situation, we find that making $\ep$ go to $0$
yields 
\[c_1(\F) \cdotp (\pi^*H+tA)^{n-1} \le \left[n(t+C\delta) A+ (1-\delta) D\right] \cdotp (\pi^*H+tA)^{n-1}\]
As $D$ is $\pi$-exceptional, it suffices to make $t$ and $\delta$ going to zero to obtain
\[c_1(\F) \cdotp (\pi^*H)^{n-1} \le 0\]
which ends the proof in the case where $K_X$ is nef.\\

The case where $-K_X$ is nef is very similar. Again, it is enough to show that every coherent subsheaf $\F \subset \Omega_X^1$ has non-positive slope with respect to any "polarization" $\om^{n-1}$ pulled-back from $Y$. 
We solve the same Monge-Ampère equation but now the volume form $dV_t$ satisfies $\Ric dV_t = \om_t-\sum a_i \Theta_i$, where $\om_t\in c_1(-\pi^*K_Y)$ satisfies $\om_t \ge -t \om_A$. We get a metric $\om$ satisfying
$\Ric \om = \om_t -(1-\delta)\theta-\delta \Theta$. Now, $\Theta(\Omega_X^1)= - \Theta(T_X)^*$, so that $\tr_{\om} \Theta(\Omega_X^1) = -\sharp \Ric$. We can run the same computations as above now, and get the non-positivity of the 
slope of $\F$. 
\end{proof}

We can now move on to the proof of the polystability of the tangent sheaf of singular Calabi-Yau varieties:

\begin{proof}[Proof of Corollary \ref{cy}]
Let us start with a "polarization" $\om_Y$, and take the notations of the proof above where we work on a resolution $\pi:X\to Y$ of the Calabi-Yau variety $Y$, and call $E$ the exceptional divisor of $\pi$. As the slope of $\mathscr T_Y$ vanishes, Theorem \ref{thm:gsp} gives the $\om$-semistability of $\mathscr T_Y$ already. Now, suppose that there exists a reflexive subsheaf $\F\subset \mathscr T_Y$ whose slope with respect to $\om_Y$ vanishes too. As in the case $K_Y$ ample, Step 4, one chooses $\F$ of minimal rank, and it suffices to find an holomorphic complement of $\F$ inside $\mathscr T_Y$. Therefore we are reduced to showing that one can complement over $X\setminus E$ any reflexive subsheaf $\mathscr G \subset \mathscr T_X$  whose slope with respect to $\pi^*\om$ vanishes. 

So we run the proof of Theorem \ref{thm:gsp} above, and as $K_Y$ is numerically trivial, one can choose $\om_t:=0$ from the very beginning (and $\delta=0$). So we end up with a family of Kähler metrics $\om:=\om_{t,\ep}\in c_1(\pi^*H+tA)$ solving $\Ric \om =\theta_{\ep}$ where $\theta_{\ep}=  \sum_{i\in I}a_i\left(\frac{ \ep ^2 |D's_i|^2}{(|s_i|^2+\ep^2)^2}+\frac{\ep^2 \Theta_i}{|s|^2+\ep^2}\right)$ is an smooth approximant of the current of integration along $D$, $D$ being defined as the exceptional divisor representing $K_X$. 

Let $\beta_{t,\ep}$ is the second fundamental form of $F\subset T_X$ with respect to the hermitian metric $h_{t,\ep}$ induced by $\om$ over $X\setminus W$. We know (from \cite{EGZ} e.g.) that $\om_{t,\ep}$ converges in $\mathscr C^{\infty}_{\rm loc}(X\setminus D)$ to a smooth Ricci flat metric $\omi$ inducing a smooth hermitian metric $h_{\infty}$ on $T_X$ with respect to which the second fundamental form of $F$ is $\bi$ the smooth limit of the $\beta_{t,\ep}$. Therefore the computations above combined with Fatou lemma show that
\[\int_{X\setminus(W\cup D)} \tr_{\rm End}(\bi^*\wedge \bi \wedge \omi^{n-1}) = 0\]
i.e. $\bi=0$ so that we find $\mathscr T_X = \G\oplus \G^{\perp}$ over $X\setminus(W\cup D)$. We can push forward this identity to $Y$, which by reflexivity gives us the expected complement of $\F$ inside $\mathscr T_X$.
\end{proof}

\bibliographystyle{smfalpha}
\bibliography{biblio.bib}

\end{document}